\documentclass{article}
\usepackage{latexsym,amssymb,amsmath,amsthm}
\newtheorem{theo}{Theorem}[section]
\newtheorem{coro}[theo]{Corollary}
\newtheorem{lemm}[theo]{Lemma}
\newtheorem{prop}[theo]{Proposition}

\newtheorem{rema}[theo]{Remark}

\author{I.D. Chipchakov\footnote{Partially supported by Grant MI-1503/2005 of
the Bulgarian National Science Fund} \\Institute of Mathematics and
Informatics,\\Bulgarian Academy of Sciences,\\
Acad. G. Bonchev Str., bl. 8, 1113, Sofia, Bulgaria}
\title{On Henselian valuations and Brauer groups of primarily quasilocal
fields}
\date{}
\begin{document}
\maketitle

\bigskip
\centerline{\it Dedicated to Professor Serban Basarab,}
\smallskip
\centerline{\it on the occasion of his seventieth birthday}

\begin{abstract} This paper finds a classification, up-to an isomorphism, of
abelian torsion groups realizable as Brauer groups of major types of
Henselian valued primarily quasilocal fields with totally
indivisible value groups. When $E$ is a quasilocal field with such a
valuation, it shows that the Brauer group of $E$ is divisible and
embeddable in the quotient group of the additive group of rational
numbers by the subgroup of integers.
\end{abstract}

\section{Introduction and statement of the main result}

A field $K$ is said to be primarily quasilocal (abbr, PQL), if every
cyclic extension $F$ of $K$ is embeddable as a subalgebra in each
central division $K$-algebra $D$ of Schur index ind$(D)$ divisible
by the degree $[F\colon K]$; we say that $K$ is quasilocal, if its
finite extensions are PQL-fields. This paper is devoted to the study
of the Brauer group Br$(K)$ when $K$ is PQL and possesses a
Henselian valuation $v$. It determines the structure of the
$p$-component Br$(K) _{p}$ of Br$(K)$, for a given prime number $p$,
under the hypothesis that the value group $v(K)$ of $(K, v)$ is
$p$-indivisible, i.e. $v(K) \neq pv(K)$. This enables us to describe
the isomorphism classes of Brauer groups of Henselian PQL-fields
with totally indivisible value groups (i.e. $p$-indivisible, for
each prime $p$), and to do the same in the special case where the
considered valued fields are quasilocal. The method of proving our
main results makes it possible to establish the existence of new
types of Henselian real-valued quasilocal fields, which make
interest in the context of the recently posed problem of
characterizing central division algebras over finitely-generated
fields $F$ by their splitting fields of finite degree over $F$ (see
Proposition 6.5 and the comment on Remark 6.6).

\smallskip
The basic notation, terminology and conventions kept in this paper
are standard and virtually the same as in \cite{Ch5}, I, and
\cite{Ch6}. Throughout, Brauer and value groups are written
additively, Galois groups are viewed as profinite with respect to
the Krull topology, and by a profinite group homomorphism, we mean a
continuous one. As usual, $\mathbb Q /\mathbb Z$ denotes the
quotient group of the additive group of rational numbers by the
subgroup of integers. We write $\mathbb P$ for the set of prime
numbers, and for each $p \in \mathbb P$, $\mathbb F_{p}$ is a field
with $p$ elements, $\mathbb Z_{p}$ is the additive group of $p$-adic
integers and $\mathbb Z(p ^{\infty })$ is the quasicyclic $p$-group.
For any profinite group $G$, we denote by cd$(G)$ the cohomological
dimension of $G$, and by cd$_{p}(G)$ its cohomological
$p$-dimension, for each $p \in \mathbb P$. Given a field $E$, $E
_{\rm sep}$ denotes a separable closure of $E$, $\mathcal{G}_{E} =
\mathcal{G}(E _{\rm sep}/E)$ is the absolute Galois group of $E$,
$\Pi (E) = \{p \in \mathbb P\colon \ {\rm cd}_{p}(\mathcal{G}_{E})
\neq 0\}$ and $P(E)$ is the set of those $p \in \mathbb P$, for
which $E$ is properly included in its maximal $p$-extension $E(p)$
in $E _{\rm sep}$. In what follows, for any $p \in P(E)$, $r(p)_{E}$
denotes the rank of $\mathcal{G}(E(p)/E)$, i.e. the cardinality of
any minimal system of generators of $\mathcal{G}(E(p)/E)$ as a
profinite group; we put $r(p)_{E} = 0$ in case $p \notin P(E)$. We
write $s(E)$ for the class of finite-dimensional central simple
$E$-algebras, $d(E)$ stands for the class of division algebras $D
\in s(E)$, and for each $A \in s(E)$, $[A]$ is the similarity class
of $A$ in Br$(E)$. For any field extension $E ^{\prime }/E$, we
denote by $I(E ^{\prime }/E)$ the set of its intermediate fields,
and by $\rho _{E'/E}$ the scalar extension map of Br$(E)$ into Br$(E
^{\prime })$. When $E ^{\prime }/E$ is finite and separable,
Cor$_{E'/E}$ denotes the corestriction homomorphism of Br$(E
^{\prime })$ into Br$(E)$. For convenience of the reader, we recall
that $E$ is said to be stable, if each $D \in d(E)$ has exponent
exp$(D)$ equal to ind$(D)$; we say that $E$ is absolutely stable, if
its finite extensions are stable fields. The field $E$ is called
$p$-quasilocal, for some $p \in \mathbb P$, if one of the following
conditions holds: (i) Br$(E) _{p} \neq \{0\}$ or $p \notin P(E)$;
(ii) every extension of $E$ in $E(p)$ of degree $p$ is embeddable as
an $E$-subalgebra in each $\Delta _{p} \in d(E)$ of index $p$. By
\cite{Ch5}, I, Theorem~4.1, $E$ is PQL if and only if it is
$p$-quasilocal, for each $p \in P(E)$. In this paper, we use at
crucial points the following characterization of the $p$-quasilocal
property (which is obtained as a consequence of \cite{Ch5}, I,
Theorems~3.1 (i)-(ii) and 4.1, and the general
restriction-corestriction (abbr, RC) formula for Brauer groups, see,
e.g., \cite{Ti}, Theorem~2.5):
\par
\medskip
(1.1) A field $E$ is $p$-quasilocal, for some $p \in \mathbb P$, if and
only if Cor$_{M/E}$ maps Br$(M) _{p}$ injectively into Br$(E) _{p}$,
for each finite extension $M$ of $E$ in $E(p)$. When this is
the case and Br$(E) _{p}$ is divisible, Cor$_{M/E}$ maps Br$(M)
_{p}$ bijectively upon Br$(E) _{p}$, for every $M$ of the considered
type.
\par
\medskip
The present research is naturally incorporated in the study of
Brauer groups of the basic types of stable fields. This problem has
two major aspects. In the first place, the structure of Br$(F)$ of
stable fields $F$ makes interest in the context of index-exponent
relations in central simple algebras over arbitrary fields (cf.
\cite{P}, Sects. 14.4 and 19.6). In the absolutely stable case, the
discussed problem is also related to the study of cohomological
properties of $\mathcal{G}_{E}$ (see \cite{P}, Sect. 14.6, and
\cite{Ch5}, I, Theorem~8.1). Secondly, the description of Br$(L)$,
for a given stable field $L$, usually reflects adequately an
essential part of the specific nature of $L$. Note in this
connection that important classes of stable fields $L$ have been
singled out by analyzing special properties of $L$. In particular,
this applies to the absolute stability of global fields (cf.
\cite{Re}, (32.19), function fields of algebraic curves defined over
a PAC-field \cite{E2}, function fields of algebraic surfaces over an
algebraically closed field of zero characteristic \cite{Jo} (see
also \cite{Li1}), and quasilocal fields \cite{Ch5}, I,
Proposition~2.3. In these cases, cd$(\mathcal{G}_{L(\sqrt{-1})}) \le
2$ (cf. \cite{S1}, Ch. II, 3.3 and 4.1, and \cite{Ch2}, I, Sect. 4),
which ensures that Br$(L)$ is a divisible group unless $L$ is
formally real (see \cite{S1}, Ch. II, 2.3, and \cite{Dr1}, page
110). The study of the stability property in the class of Henselian
fields $(K, v)$ sheds new light on the considered problem. As it
turns out, the residue field $\widehat K$ of $(K, v)$ is PQL
whenever $K$ is stable and the value group $v(K)$ is totally
indivisible \cite{Ch5}, I, Proposition~2.1. Moreover, these
conditions frequently ensure that $\widehat K$ is almost perfect,
i.e. its finite extensions have primitive elements (cf. \cite{Ch1},
Theorem~2.1, and \cite{Ch5}, I, (1.8) and Proposition~2.3). The
relations between $v(K)$ and $\widehat K$ make it possible to
characterize basic classes of Henselian stable fields (see
Proposition 2.2 and \cite{Ch1}, Theorem~3.1 and Sect. 4). They also
show explicitly how Brauer and character groups of PQL-fields
determine the structure of Brauer groups of stable fields (see
Remark 2.4 and \cite{Ch5}, II).
\par
\medskip
Brauer groups of PQL-fields also have strong influence on the study
of the norm groups of their finite abelian extensions. Specifically,
this applies to the relations described by the second part of the
following assertion (cf. \cite{Ch6}, Theorem~3.1, and \cite{Ch5}, I,
Lemma~4.2 (ii)):
\par
\medskip
(1.2) Let $E$ be a $p$-quasilocal field, $\Omega _{p}(E)$ the set of
finite abelian extensions of $E$ in $E(p)$, Nr$(E)$ the set of norm
groups of $E$, and $_{p} {\rm Br}(E) = \{b \in {\rm Br}(E)\colon \
pb = 0\}$. Then:
\par
(i) The natural mapping of $\Omega _{p}(E)$ into Nr$(E)$ (by the
rule $M \to N(M/E)$, $M \in \Omega _{p}(E)$) is injective, and for
each $M _{1}$, $M _{2} \in \Omega _{p}(E)$, the norm group (over
$E$) of the compositum $M _{1}M _{2}$ equals the intersection $N(M
_{1}/E) \cap N(M _{2}/E)$, and $N(M _{1} \cap M _{2}/E) = N(M
_{1}/E)N(M _{2}/E)$.
\par
(ii) For each $M \in \Omega _{p}(E)$, the quotient group $E ^{\ast
}/N(M/E)$ decomposes into a direct sum $\mathcal{G}(M/E) ^{d(p)}$ of
isomorphic copies of the Galois group $\mathcal{G}(M/E)$, indexed by
a set of cardinality $d(p)$, the dimension of $_{p}{\rm Br}(E)$ as a
vector space over $\mathbb F_{p}$. In particular, if Br$(E) _{p} =
\{0\}$, then $N(M/E) = E ^{\ast }$.
\par
\medskip\noindent
When $E$ is a PQL-field and $L/E$ is a finite abelian extension, it
follows from (1.2) (ii) and \cite{Ch6}, Lemma~2.1, that $E ^{\ast
}/N(L/E)$ is isomorphic to the direct product of the groups $E
^{\ast }/N(L _{p}/E)\colon \ p \in P(E), p \mid [L\colon E]$, where
$L _{p} = L \cap E(p)$, for each admissible $p$. This is an analogue
to the local reciprocity law whose form is determined by the
sequence $d(p)\colon \ p \in P(E)$, defined in (1.2) (ii). It is
therefore worth noting that an abelian torsion group is isomorphic
to Br$(\Phi )$, for some PQL-field $\Phi = \Phi (T)$ if and only if
one of the following conditions holds (see \cite{Ch7}, Sect. 1 and
Proposition~6.4):
\par
\medskip
(1.3) (i) $T$ is divisible; then $\Phi $ is nonreal and can be
chosen among those quasilocal fields, for which the maps $\rho
_{\Phi /\Phi '}$, $\Phi ^{\prime } \in I(\Phi _{\rm sep}/\Phi )$,
are surjective;
\par
(ii) The $2$-component $T _{2}$ is of order $2$ and the
$p$-components $T _{p}$, $p \in (\mathbb P \setminus \{2\})$, are
divisible; in this case, $\Phi $ is formally real.
\par
\medskip
Since the notion of a quasilocal field extends the one of a local
field and defines a class containing such frequently used
representatives as $p$-adically closed fields and Henselian discrete
valued fields with quasifinite residue fields (cf. \cite{S2}, Ch.
XIII, Sect. 3, and \cite{PR}, Theorem~3.1 and Lemma~2.9), these
facts attract interest in the role of Henselian valuations for
arbitrary quasilocal fields. The main results of this paper, stated
below, enable one to evaluate this role by comparing (1.3) with the
structure of Br$(K)$ when $(K, v)$ is a Henselian quasilocal field,
such that $v(K)$ is totally indivisible (see also Corollaries 5.3
and 5.4):

\medskip
\begin{theo}
Let $(K, v)$ be a Henselian $p$-quasilocal field with $v(K) \neq
pv(K)$, for some $p \in \mathbb P$. Then:
\par
{\rm (i)} {\rm Br}$(K) _{p}$ is trivial or isomorphic to $\mathbb
Z(p ^{\infty })$ except, possibly, in the case where $r(p)_{K} = 1$,
{\rm char}$(K) \neq p$ and $K$ does not contain a primitive $p$-th
root of unity;
\par
{\rm (ii)} $K$ is subject to the following alternative relative to
$v$:
\par
{\rm ($\alpha $)} There exists a $\mathbb Z_{p}$-extension $I _{\infty }$
of $K$ in $K(p)$, such $v _{\infty }(I _{\infty }) = v(K)$ and the
residue field of $(I _{\infty }, v _{\infty })$ is separable over
$\widehat K$, where $v _{\infty }$ is the unique, up-to an
equivalence, valuation of $I _{\infty }$ extending $v$;
\par
{\rm ($\beta $)} Finite extensions of $K$ in $K(p)$ are totally
ramified;
\par
{\rm (iii)} When $p \in P(K)$, {\rm Br}$(K) _{p} = \{0\}$ if and
only if finite extensions of $K$ in $K(p)$ are totally ramified and
the group $v(K)/pv(K)$ is of order $p$.
\end{theo}

\medskip
When $p \neq {\rm char}(\widehat K)$ and Br$(K) _{p} \neq \{0\}$,
the isomorphism Br$(K) _{p} \cong \mathbb Z(p ^{\infty })$ is established
in Section 4 by proving the following assertion:
\par
\medskip
(1.4) If $p \neq {\rm char}(\widehat K)$ and $K$ contains a
primitive $p$-th root of unity, then $\mathcal{G}(K(p)/K)$ is a
Demushkin group (in the sense of \cite{S1}) with $r(p)_{K} = 2$ or
is isomorphic to $\mathbb Z_{p}$ depending on whether or not Br$(K) _{p}
\neq \{0\}$.
\par
\medskip\noindent
The proof of Theorem 1.1 (i) in the case where char$(\widehat K) =
p$ and there exists an immediate cyclic extension $I/K$ of degree
$p$ is presented in Section 3 (the realizability of this special
case is demonstrated by Proposition 6.2). This part of the proof is
based on the divisibility of Br$(K) _{p}$ (see Lemma 3.3 (i)) as
well as on (1.2) (i) and general properties of Henselian valuations
and isolated subgroups of their value groups. The rest of the proof
of Theorem 1.1 is contained in Section 4. When $p = {\rm
char}(\widehat K)$, we adapt to our setting the proof of \cite{TY},
Theorem~3.1. Our argument also relies on (1.2) and on the method of
proving the main results of \cite{Ch2}, I. The remaining part of the
paper presents consequences of the main result. In Section 5, we
describe the isomorphism classes of Brauer groups of Henselian
PQL-fields $(K, v)$ such that $v(K)$ is totally indivisible (see
Corollary 5.5, (5.2) and (5.3)). When $K$ is quasilocal, we also
prove the cyclicity of every $D \in d(K)$ (see Corollary 5.3). In
Section 6, we complete the characterization of the quasilocal
property in the class of Henselian fields with totally indivisible
value groups, started in \cite{Ch2}, I; also, we give a criterion
for divisibility of Brauer groups of quasilocal fields, and for
defectlessness of their finite separable extensions.

\section{Preliminaries on Henselian valuations and completions}

Let $K$ be a field with a nontrivial (Krull) valuation $v$, $O
_{v}(K) = \{a \in K\colon \ v(a) \ge 0\}$ the valuation ring of $(K,
v)$, $M _{v}(K) = \{\mu \in K\colon \ v(\mu ) > 0\}$ the unique
maximal ideal of $O _{v}$, $v(K)$ and $\widehat K$ the value group
and the residue field of $(K, v)$, respectively, Is$_{v} ^{\prime }
(K)$ the set of isolated subgroups of $v(K)$ and Is$_{v}(K) = {\rm
Is}_{v} ^{\prime }(K) \setminus \{v(K)\}$. It is well-known that,
for each $H \in {\rm Is}_{v}(K)$, the ordering of $v(K)$ induces
canonically on $v(K)/H$ a structure of an ordered group, and one can
naturally associate with $v$ and $H$ a valuation $v _{H}$ of $K$
with $v _{h} (K) = v(K)/H$. Unless specified otherwise, $K _{H}$
will denote the residue field of $(K, v _{H})$, $\eta _{H}$ the
natural projection $O _{v _{H}}(K) \to K _{H}$, and $\hat v _{H}$
the valuation of $K _{H}$ induced canonically by $v$ and $H$. The
valuations $v$, $v _{H}$ and $\hat v _{H}$ are related as follows
(see \cite{E3}, Proposition~5.2.1):
\par
\medskip
(2.1) (i) $\hat v _{H}(K _{H}) = H$, $\widehat K _{H}$ is isomorphic
to $\widehat K$ and $\eta _{H}$ induces a surjective homomorphism of
$O _{v}(K)$ upon $O _{\hat v _{H}}(K _{H})$; when $H$ is divisible,
$v(K)$ is isomorphic to the lexicographically ordered direct sum $v
_{H}(K) \oplus \hat v _{H}(\widehat K _{H})$;
\par
(ii) If $v(K)$ properly includes the union $H(K)$ of the groups from
Is$_{v}(K)$, then $v _{H(K)}$ is real-valued.
\par
\medskip
Recall further that the topology of $K$ induced by $v _{H}$ does not
depend on the choice of $H$ and the mapping of Is$_{v}(K)$ on the
set $V _{v}$ of subrings of $K$ including $O _{v}$, defined by the
rule $X \to O _{v _{X}}(K)$, $X \in {\rm Is}_{v}(K)$, is an
inclusion-preserving bijection. By H\"older's theorem (cf.
\cite{E3}, Theorem~2.5.2), Is$_{v}(K) = \{0\}$ if and only if $v(K)$
is Archimedean, i.e. it embeds as an ordered subgroup in the
additive group $\mathbb R$ of real numbers. When this is the case, we
identify $v(K)$ with its isomorphic copy in $\mathbb R$.
\par
We say that $(K, v)$ is Henselian, if the valuation $v$ is
Henselian, i.e. $v$ is uniquely, up-to an equivalence, extendable to
a valuation $v _{L}$ on each algebraic field extension $L/K$. In
order that $v$ is Henselian, it is necessary and sufficient that the
Hensel-Rychlik condition holds (cf. \cite{E3}, Sect. 18.1):
\par
\medskip
(2.2) Given a polynomial $f(X) \in O _{v}(K) [X]$, and an element $a
\in O _{v}(K)$, such that $2v(f ^{\prime }(a)) < v(f(a))$, where $f
^{\prime }$ is the formal derivative of $f$, there is a zero $c \in
O _{v}(K)$ of $f$ satisfying the equality $v(c - a) = v(f(a)/f
^{\prime }(a))$.
\par
\medskip
When $v(K)$ is not Archimedean, the Henselian property can be also
characterized as follows (see, e.g., \cite{Ch8}, Sect. 2):
\par
\medskip
\begin{prop} Let $(K, v)$ be a valued field, and let $H \in {\rm
Is}_{v}(K)$. Then $v$ is Henselian if and only if $v _{H}$ and $\hat
v _{H}$ are Henselian.
\end{prop}

\medskip
When $v$ is Henselian and $L/K$ is an algebraic extension, $v _{L}$
is also Henselian and extends uniquely to a valuation $v _{D}$ on
each $D \in d(L)$. Denote by $\widehat D$ the residue field of $(D,
v _{D})$, put $v(D) = v _{D}(D)$, and let $e(D/K)$ be the
ramification index of $D/K$, i.e. the index of $v(K)$ in $v(D)$. By
the Ostrowski-Draxl theorem \cite{Dr2}, $[D\colon K]$, $[\widehat
D\colon \widehat K]$ and $e(D/K)$ are related as follows:
\par
\medskip
(2.3) $[D\colon K]$ is divisible by $[\widehat D\colon \widehat
K]e(D/K)$ and $[D\colon K]/([\widehat D\colon \widehat K]e(D/K))$ is
not divisible by any $p \in \mathbb P$, $p \neq {\rm char}(\widehat K)$.
\par
\medskip
The $K$-algebra $D$ is said to be defectless, if $[D\colon K] =
[\widehat D\colon \widehat K]e(D/K)$, and it is called totally
ramified, if $e(D/K) = [D\colon K]$. The following lemma, proved in
\cite{Ch8}, Sect. 6, characterizes the case in which $v(K) \neq
pv(K)$ and $I(K(p)/K)$ does not contain totally ramified extensions
of $K$, for a given $p \in \mathbb P$.

\medskip
\begin{lemm} Let $(K, v)$ be a Henselian field with
$v(K) \neq pv(K)$, for some $p \in \mathbb P$. Then $K$ is subject
to the following alternative relative to $v$:
\par
{\rm (i)} There exists a field $\Phi \in I(K(p)/K)$, such that
$[\Phi \colon K] = p$ and $\Phi /K$ is totally ramified;
\par
{\rm (ii)} {\rm char}$(K) = 0$, $K$ does not contain a primitive
$p$-th root of unity and the minimal group from {\rm Is}$_{v}
^{\prime }(K)$ containing $v(p)$ is $p$-divisible.
\end{lemm}

\medskip
Let $v$ be Henselian and ${\rm In}(K)$ be the class of inertial
$K$-algebras (i.e. those division $K$-algebras $D$, for which
$[D\colon K] = [\widehat D\colon \widehat K] \in \mathbb N$ and $\widehat
Z/\widehat K$ is a finite separable extension, where $\widehat Z$ is
the centre of $\widehat D$). Then:
\par
\medskip
(2.4) (i) For each finite-dimensional division $\widehat K$-algebra
$\widetilde \Delta $ whose centre is separable over $\widehat K$,
there exists $\Delta \in {\rm In}(K)$ with $\widehat \Delta \cong
\widetilde \Delta $ over $\widehat K$; $\Delta $ is uniquely
determined by $\widetilde \Delta $, up-to a $K$-isomorphism
\cite{JW}, Theorem~2.8 (a) (and is called an inertial lift of
$\widetilde D$ over $K$).
\par
(ii) In order that an inertial field extension $Y/K$ is a Galois
extension, it is necessary and sufficient that $\widehat Y/\widehat
K$ is Galois; when this occurs, the Galois groups $\mathcal{G}(Y/K)$
and $\mathcal{G}(\widehat Y/\widehat K)$ are canonically isomorphic
(cf. \cite{JW}, page 135).
\par
(iii) The set IBr$(K) = \{[I]\colon \ I \in {\rm In}(K)\}$ forms a
subgroup of Br$(K)$, which is canonically isomorphic to Br$(\widehat
K)$ (cf. \cite{JW}, Theorem~2.8 (b)).

\medskip
Assuming as above that $v$ is Henselian, let $D$ be an arbitrary
finite-dimensional division $K$-algebra, $D _{v}$ a completion of
$D$ with respect to the topology induced by $v$, and $Z(D)$ the
centre of $D$. It is known that there is a close relationship
between finite-dimensional division $K$-algebras and the
corresponding $K _{v}$-algebras, described in part by the following
statement:
\par
\medskip
(2.5) (i) $K$ is separably closed in $K _{v}$ and the valuation
$\bar v$ of $K _{v}$ continuously extending $v$ is Henselian;
\par
(ii) The natural mapping of $D \otimes _{K} K _{v}$ into $D _{v}$ is
a $K _{v}$-isomorphism whenever $Z(D)/K$ is separable; hence, $\rho
_{K/K _{v}}$ is injective and preserves indices and exponents;
\par
(iii) A finite extension $L$ of $K$ in $K _{\rm sep}$ embeds in an
algebra $U \in d(K)$ if and only if $L _{v}$ embeds in $U _{v}$ over
$K _{v}$.

\par
\medskip
Note also the following characterization of finite extensions of $K
_{v}$ in $K _{v,{\rm sep}}$, in case $v$ is Henselian (cf.
\cite{B1}, Ch. VI, Sect. 8, No 2, and \cite{JW}, page 135):
\par
\medskip
(2.6) (i) Every finite extension $L$ of $K _{v}$ in $K _{v,{\rm
sep}}$ is $K _{v}$-isomorphic to $\widetilde L \otimes _{K} K _{v}$
and $\widetilde L _{v}$, where $\widetilde L$ is the separable
closure of $K$ in $L$; $L/K _{v}$ is Galois if and only if
$\widetilde L/K$ is Galois; when this holds, $\mathcal{G}(L/K _{v})
\cong \mathcal{G}(\widetilde L/K)$ (canonically).
\par
(ii) $K _{\rm sep} \otimes _{K} K _{v}$ is a field, $K _{\rm sep}
\otimes _{K} K _{v} \cong K _{v,{\rm sep}}$ and $\mathcal{G}_{K}
\cong \mathcal{G}_{K _{v}}$.
\par
\medskip
Statements (2.5) reduce the study of index-exponent relations in
central division algebras over a Henselian field $(K, v)$ to the
special case in which $K = K _{v}$. Combining, for instance, (2.5)
and (2.6) (i) with \cite{Ya}, Proposition~2.1, and \cite{Ch1},
Theorem~3.1, one obtains the following result:

\medskip
\begin{prop} Let $(K, v)$ be a Henselian discrete
valued field. Then:
\par
{\rm (i)} $K$ is stable, provided that $\widehat K$ is almost
perfect, stable and PQL.
\par
{\rm (ii)} $K$ is absolutely stable if and only if $\widehat K$ is
quasilocal and almost perfect.
\end{prop}

\par
\medskip
\begin{rema} Statement (1.3) and the method of proving it in
\cite{Ch7}, combined with Proposition 2.2 and Scharlau's
generalization of Witt's decomposition theorem \cite{Sch}, make it
possible to study effectively the structure of the Brauer groups of
various types of stable fields (see also \cite{Ch5}, II, Sect. 3 and
Lemma~2.3, for the relations between Br$(E)$ and the character group
of $\mathcal{G}_{E}$). Specifically, they enable one to find series
of absolutely stable fields $F$ with Henselian valuations and
indivisible groups Br$(F)$, for infinitely many $p \in \mathbb P$
(see \cite{Ch1}, Corollary~4.7, and \cite{Ch7}, Proposition~6.8).
\end{rema}

\medskip
Let us note that, for every Henselian field $(K, v)$ with char$(K) =
p$ and $v(K) \neq pv(K)$, $r(p)_{K} = \infty $, i.e.
$\mathcal{G}(K(p)/K)$ is not finitely-generated as a pro-$p$-group
(cf. \cite{Ch8}, Remark~4.2). Therefore, the following two
propositions (proved in \cite{Ch8}, and generalizing \cite{Po},
(2.7)) fully characterize the case where $r(p)_{K} \in \mathbb N$.
The characterization depends on whether or not the minimal group
$G(K) \in {\rm Is}_{v}^{\prime }(K)$ containing $v(p)$ is
$p$-divisible.

\medskip
\begin{prop} Let $(K, v)$ be a Henselian
field with {\rm char}$(K) = 0$, {\rm char}$(\widehat K) = p > 0$ and
$G(K) = pG(K)$, and let $\varepsilon $ be a primitive $p$-th root of
unity in $K _{\rm sep}$. Then $r(p)_{K} \in \mathbb N$ if and only
if $r(p)_{K _{G(K)}} \in \mathbb N$ and one of the following
conditions holds:
\par
{\rm (i)} $v(K) = pv(K)$ or $\varepsilon \notin K$; in this case,
finite extensions of $K$ in $K(p)$ are inertial relative to $v
_{G(K)}$ and $\mathcal{G}(K(p)/K) \cong \mathcal{G}(K _{G(K)}(p)/K
_{G(K)})$;
\par
{\rm (ii)} $\varepsilon \in K$ and $v(K)/pv(K)$ is of order $p
^{\tau }$, for some $\tau \in \mathbb N$; in this case,
$\mathcal{G}(K(p)/K)$ is isomorphic to a topological semi-direct
product
\par\noindent
$\mathbb Z_{p} ^{\tau } \times \mathcal{G}(K _{G(K)}(p)/K
_{G(K)})$.
\par\noindent
When $r(p)_{K} \in \mathbb N$, $\widehat K$ is perfect, $\mathcal{G}(K
_{G(K)}(p)/K _{G(K)})$ is a trivial or a free pro-$p$-group, and
either $p \in P(\widehat K)$ or $K _{G(K)}(p)/K _{G(K)}$ is
immediate relative to $\hat v _{G(K)}$.
\end{prop}

The concluding part of our next result is implied by (2.3),
\cite{TY}, Proposition~2.2.

\medskip
\begin{prop} Let $(K, v)$ be a Henselian field with {\rm char}$(K) =
0$ and {\rm char}$(\widehat K) = p \neq 0$. Then the following
conditions are equivalent:
\par
{\rm (i)} $G(K) \neq pG(K)$ and $r(p)_{K} \in \mathbb N$;
\par
{\rm (ii)} $\widehat K$ is finite, $G(K)$ is cyclic, and in case $K$
contains a primitive $p$-th root of unity, the group $v
_{G(K)}(K)/pv _{G(K)}(K)$ is finite.
\par\noindent
When these conditions hold, finite extensions of $K$ in $K _{\rm
sep}$ are defectless.
\end{prop}

\medskip
\begin{coro} Assume that $(K, v)$ is a Henselian $p$-quasilocal field, such that {\rm char}$(\widehat K) = p$,
$r(p)_{K} \ge 2$ and $v(K)/pv(K)$ is noncyclic. Then:
\par
{\rm (i)} $v(K)/pv(K)$ is of order $p ^{2}$ and finite extensions of
$K$ in $K(p)$ are totally ramified;
\par
{\rm (ii)} $K$ contains a primitive $p ^{n}$-th root of unity, for
each $n \in \mathbb N$, and $\mathcal{G}(K(p)/K)$ is isomorphic to
the additive group $\mathbb Z_{p} ^{2}$ with respect to its standard
topology.
\end{coro}

\begin{proof}
Denote by $\Sigma _{p}(K)$ the set of extensions of $K$ in $K(p)$ of
degree $p$, and by $V _{p}(K)$ the set of subgroups of $v(K)$
properly including $pv(K)$ and different from $v(K)$. Let $K _{1}$
and $K _{2}$ be different elements of $\Sigma _{p}(K)$. Then (1.2)
(i) yields $N(K _{1}/K)N(K _{2}/K) = K ^{\ast }$, which means that
$v(K) = pv(K _{1}) + pv(K _{2})$. In other words, it follows from
(2.3) and the noncyclicity of $v(K)/pv(K)$ that $v(K)/pv(K)$ is of
order $p ^{2}$ and the mapping of $\Sigma _{p}(K)$ into $V _{p}(K)$
by the rule $L \to pv(L)$, $L \in \Sigma _{p}(K)$, is well-defined
and bijective. This observation proves that $r(p)_{K} = 2$ and every
$L \in \Sigma _{p}(K)$ is totally ramified over $K$. It is now easy
to see from (2.4) (i) that $p \notin P(\widehat K)$. This implies
$\widehat K$ is infinite, so it follows from Proposition 2.6 and
\cite{Ch8}, Remark~4.2, that char$(K) = 0$ and $G(K) = pG(K)$.
Applying Proposition 2.5, one obtains further that $\widehat K$ is
perfect and $p \notin P(K _{G(K)})$. Since char$(K _{G(K)}) = 0$, $v
_{G(K)}$ is Henselian and $p \notin P(K _{G(K)})$, this enables one
to deduce from (2.3) that finite extensions of $K$ in $K(p)$ are
totally ramified relative to $v _{G(K)}$ (and because of the
equality $G(K) = pG(K)$, they have the same property relative to
$v$). In view of \cite{Ch3}, Lemma~1.1, these observations show that
$K$ contains a primitive $p ^{n}$-th root of unity, for each $n \in
\mathbb N$. As $v(K)/pv(K)$ is of order $p ^{2}$, it is now easy to
see from Proposition 2.5 that $\mathcal{G}(K(p)/K) \cong \mathbb Z
_{p} ^{2}$, which completes the proof of Corollary 2.7.
\end{proof}

\medskip
\begin{rema} It is known that if $(K, v)$ is a Henselian field satisfying the conditions of Proposition 2.6, then
Br$(K _{G(K)})$ $\cong \mathbb Q /\mathbb Z$ and Br$(K _{G(K)})$
embeds in Br$(K)$ (see (2.4) (iii) and \cite{S2}, Ch. XII, Sect. 3).
Also, it follows from (2.1) (i), (2.4) (ii), (2.6) and \cite{S1},
Ch. II, Theorems~3 and 4, that $r(p)_{K_{G(K)}} \ge 2$.
\end{rema}

\section{On the Brauer group of a Henselian $p$-quasilocal
field with a $p$-indivisible value group}

In this Section we prove that if $(K, v)$ is a Henselian
$p$-quasilocal field satisfying the conditions of Theorem 1.1, and
if $K$ possesses an immediate extension in $K(p)$ of degree $p$,
then Br$(K) _{p}$ is isomorphic to $\mathbb Z(p ^{\infty })$.

\medskip
\begin{theo} Under the hypotheses of Theorem 1.1,
suppose that $K(p)$ contains as a subfield an immediate extension
$I$ of $K$ of degree $p$. Then $\widehat K$ is perfect, $v(K)/pv(K)$
is of order $p$ and $\nabla _{0}(K) \subset N(I/K)$.
\end{theo}

The proof of Theorem 3.1 relies on the following two lemmas lemma,
the first of which has been proved in \cite{Ch8}.

\medskip
\begin{lemm} In the setting of Lemma 2.2, suppose that $\widehat K$ is
imperfect and $W _{p}(K)$ is the set of those $\Lambda \in
I(K(p)/K)$, for which $[\Lambda \colon K] = [\widehat \Lambda \colon
\widehat K] = p$ and $\widehat \Lambda $ is purely inseparable over
$\widehat K$. Then:
\par
{\rm (i)} $W _{p}(K)$ is infinite except, possibly, in the case
where {\rm char}$(K) = 0$, $v(p) \notin pv(K)$ and $K$ does not
contain a primitive $p$-th root of unity;
\par
{\rm (ii)} When {\rm char}$(K) = 0$ and $v(p) \notin pv(K)$, there
exists a field $\Lambda ^{\prime } \in I(K(p)/K)$, such that
$[\Lambda ^{\prime }\colon K] = p$ and $v(p) \in pv(\Lambda ^{\prime
})$.
\end{lemm}

\begin{lemm} Let $(K, v)$ be a Henselian
$p$-quasilocal field with $v(K) \neq pv(K)$. Then:
\par
{\rm (i)} {\rm Br}$(K) _{p}$ is divisible;
\par
{\rm (ii)} There is at most one extension of $K$ in $K(p)$ of degree
$p$, which is not totally ramified over $K$; when such an extension
exists, Br$(K) _{p} \neq \{0\}$;
\par
{\rm (iii)} $\widehat K$ is perfect, provided that $p = {\rm
char}(\widehat K)$.
\end{lemm}

\begin{proof}
When $p > 2$, the divisibility of Br$(K) _{p}$ is a special case of
\cite{Ch5}, I, Theorem~3.1 (ii), and in case $p = 2$, it is implied
by (1.4) and the fact that $\mathcal{G}(E(2)/E)$ is a group of order
$2$, for every formally real $2$-quasilocal field $E$ \cite{Ch5}, I,
Lemma~3.5. The rest of our proof relies on the fact that if $R$ is a
finite extension of $K$ in $K(p)$, which is not totally ramified,
then $v(\lambda ) \in pv(K)$, for every $\lambda \in N(R/K)$. At the
same time, it follows from Galois theory and the normality of
maximal subgroups of finite $p$-groups (cf. \cite{L}, Ch. I, Sect.
6; Ch. VIII) that if $R \in I(K(p)/K)$ and $[R\colon K] = p$, then
$R/K$ is cyclic. Since, by (1.2) (i), $N(R _{1}/K)N(R _{2}/K) = K
^{\ast }$ whenever $R _{1}$ and $R _{2}$ lie in $I(K(p)/K)$, $R _{1}
\neq R _{2}$ and $[R _{j}\colon K] = p$, $j = 1, 2$, these
observations prove the former part of Lemma 3.3 (ii). Combined with
\cite{P}, Sect. 15.1, Proposition~b, they also imply the latter
assertion of Lemma 3.3 (ii). For the proof of Lemma 3.3 (iii), it
suffices to note that $(M, v _{M})$ satisfies the conditions of
Lemma 3.3 whenever $M \in I(K(p)/K)$ and $[M\colon K] \in \mathbb N$
(see (2.3) and \cite{Ch5}, I, Theorem~4.1 (ii)), which reduces our
concluding assertion to a consequence of Lemma 3.2 and the former
part of Lemma 3.3 (ii).
\end{proof}

\begin{rema} Let $(K, v)$ be a Henselian $p$-quasilocal field with
$v(K) \neq pv(K)$ and char$(\widehat K) \neq p$. In view of (2.3),
the former part of Lemma 3.3 (ii) can be restated by saying that
$r(p)_{\widehat K} \le 1$. Combining (2.3) and (2.4) (i) with Lemma
3.3 (i) and \cite{Ch5}, I, Lemma~3.5, one also obtains that
$\mathcal{G}(\widehat K(p)/\widehat K) \cong \mathbb Z_{p}$ unless
$p \notin P(\widehat K)$. It is therefore clear from Lemma 3.3 (ii)
and \cite{Ch3}, Lemma~1.1 (a), that if $K$ does not contain a
primitive $p$-th root of unity, then $\mathcal{G}(K(p)/K) \cong
\mathcal{G}(\widehat K(p)/\widehat K)$ and finite extensions of $K$
in $K(p)$ are inertial. Hence, by \cite{P}, Sect. 15.1, Proposition
b (and the proof of Lemma 3.3 (ii)), $_{p}{\rm Br}(K) \cong _{p}{\rm
Br}(\widehat K)$ or $_{p}{\rm Br}(K)$ is isomorphic to the direct
sum $_{p}{\rm Br}(\widehat K) \oplus v(K)/pv(K)$, depending on
whether or not $r(p)_{\widehat K} = 0$.
\end{rema}

\medskip
\begin{lemm} In the setting of Lemma 3.3, suppose that {\rm char}$(\widehat
K) = p$ and there exists a field $I \in I(K(p)/K)$, such that
$[I\colon K] = p$ and $I/K$ is not totally ramified. Then $(K, v)$
has the following properties:
\par
{\rm (i)} The group $v(K)/pv(K)$ is of order $p$, provided that
$r(p)_{K} \ge 2$; in particular, this applies to the case where
$v(K)$ is Archimedean;
\par
{\rm (ii)} A finite extension $M$ of $K$ in $K(p)$ is totally
ramified if and only if $M \cap I = K$;
\par
{\rm (iii)} $p \in P(\widehat K)$ if and only if $I/K$ is inertial;
when this holds, $\widehat K(p)/\widehat K$ is a $\mathbb Z
_{p}$-extension.
\end{lemm}

\begin{proof}
Statement (2.4) (i) and Lemma 3.3 (ii) imply the former assertion of
Lemma 3.5 (iii) and the inequality $r(p)_{\widehat K} \le 1$. As
$\widehat K$ is a nonreal field, this in turn enables one to deduce
the latter part of Lemma 3.3 (iii) from Galois theory and \cite{Wh},
Theorem~2. Note further that every $L \in I(K(p)/K)$, $L \neq K$,
contains as a subfield a cyclic extension $L _{0}$ of $K$ of degree
$p$. This well-known fact is implied by Galois theory and the
subnormality of proper subgroups of finite $p$-groups. Let now
$r(p)_{K} \ge 2$ or, equivalently, there is a field $T \in K(p)/K)$,
such that $[T\colon K] = p$ and $T \neq I$. By (1.2) (i), then
$N(T/K)N(I/K) = K ^{\ast }$, which implies that $T/K$ is totally
ramified. At the same time, it becomes clear from (2.3) that $IT/T$
is not totally ramified. Since, by \cite{Ch5}, I, Theorem~4.1 (ii),
$T$ is $p$-quasilocal, the noted properties of $T$ and $L$ make it
easy to prove Lemma 3.5 (ii), arguing by induction on $n = {\rm
log}_{p} ([M\colon K])$. The former part of Lemma 3.5 (i) follows
from Corollary 2.7 and the assumption on $I/K$, and for the proof of
the latter one, it suffices to observe that the existence of $T$ in
case $v(E) \le \mathbb R$ is guaranteed by Lemma 2.2 (and by
Proposition 2.6 and Remark 2.8).
\end{proof}

\par
\medskip
\begin{lemm} Let $(K, v)$ be a Henselian field with {\rm
char}$(\widehat K) = p > 0$, $v(K) = pv(K)$ and {\rm Br}$(K) _{p}
\neq \{0\}$. Suppose also that $K$ is $p$-quasilocal and $p \in
P(\widehat K)$. Then finite extensions of $K$ in $K(p)$ are
defectless, {\rm Br}$(K) _{p} \subseteq {\rm IBr}_{v}(K)$, {\rm
Br}$(K) _{p} \cong {\rm Br}(\widehat K) _{p}$ and $[\widehat K\colon
\widehat K ^{p}] = p$.
\end{lemm}

\begin{proof}
Statement (2.4) (i) and the assumption that $p \in P(\widehat K)$
imply the existence of an inertial extension $I _{p}$ of $K$ in
$K(p)$ of degree $p$. This ensures that $\nabla _{0}(K) \subseteq
N(I _{p}/K)$. Note further that if $\widehat K$ is perfect, then $K
^{\ast } = \nabla _{0}(K).K ^{\ast p}$, so the noted inclusion
requires that $N(I _{p}/K) = K ^{\ast }$. As Br$(K) _{p} \neq
\{0\}$, this contradicts (1.2) (i) and thereby proves that $\widehat
K \neq \widehat K ^{p}$. We show that $[\widehat K\colon \widehat K
^{p}] = p$. By Lemma 3.2, there is a field $\Lambda \in I(K(p)/K)$,
such that $[\Lambda \colon K] = [\widehat L\colon \widehat K] = p$
and $\widehat \Lambda $ is purely inseparable over $\widehat K$.
Let $\varphi $ be a generator of $\mathcal{G}(\Lambda /K)$. It is
easily seen that if $[\widehat K\colon \widehat K ^{p}] \ge p
^{2}$, then $O _{v}(K)$ contains an element $b$, such that $\hat b
\notin \widehat \Lambda ^{p}$. Therefore, the cyclic $K$-algebra $D
_{b} = (\Lambda /K, \varphi , b)$ lies in $d(K)$ and $\widehat D
_{b}/\widehat K$ is a purely inseparable field extension of degree
$p ^{2}$. This leads to the conclusion that $I _{p}$ does not embed
in $D _{b}$ as a $K$-subalgebra. Our conclusion, contradicts the
assumption that $K$ is $p$-quasilocal, which proves that $[\widehat
K\colon \widehat K ^{p}] = p$.
\par
We show that Br$(K) _{p} \subseteq {\rm IBr}(K)$ and finite
extensions of $K$ in $K(p)$ are defectless. Fix a generator $\sigma
$ of $\mathcal{G}(I _{p}/K)$ and an algebra $\Delta \in d(K)$ of
exponent $p$. As $K$ is $p$-quasilocal and, by \cite{Ch5}, I,
Theorem~3.1, ind$(\Delta ) = p$, it is easily seen that $\Delta $ is
isomorphic to the $K$-algebra $(I _{p}/K, \sigma , a)$, for some $a
\in K ^{\ast } \setminus N(I _{p}/K)$. Moreover, it follows from the
equality $v(K) = pv(K)$ that $a$ can be chosen so that $v(a) = 0$.
Hence, by the Henselian property of $v$ and the fact that $I _{p}/K$
is inertial, $\Delta /K$ is inertial too, which proves that
$_{p}{\rm Br}(K) \subseteq {\rm IBr}(K)$. Applying (2.4) (iii) and
Witt's theorem (see \cite{Dr1}, Sect. 15, and \cite{JW},
Theorem~2.8), one obtains consecutively that Br$(K) _{p} \cap {\rm
IBr}(K) \cong {\rm Br}(\widehat K) _{p}$ and Br$(K) _{p} \cap {\rm
IBr}(K)$ is a divisible subgroup of Br$(K) _{p}$. Therefore, by
\cite{F}, Theorem~24.5, Br$(K) _{p} \cap {\rm IBr}(K)$ is a direct
summand in Br$(K) _{p}$, so the inclusion $_{p}{\rm Br}(K) \subset
{\rm IBr}(K)$ implies that Br$(K) _{p} \subseteq {\rm IBr}(K)$. This
indicates that the maximal subfields of $(I _{p}/K, \sigma , a)$ are
defectless over $K$. As $K$ is $p$-quasilocal, the obtained
result proves that every $L \in I(K(p)/K)$ with $[L\colon K] = p$
is defectless over $K$. Since finite extensions of $K$ in
$K(p)$ are $p$-quasilocal, by \cite{Ch5}, I, Theorem~4.1, this
enables one to deduce from Galois theory and the normality of
maximal subgroups of finite $p$-groups that finite extensions of $K$
in $K(p)$ are defectless.
\end{proof}

\medskip
\begin{lemm} Under the hypotheses of Lemma 3.5, suppose that
$r(p)_{K} \ge 2$ and there exists a $p$-indivisible group $H \in
{\rm Is}_{v}(K)$. Then:
\par
{\rm (i)} $K _{H}$ is $p$-quasilocal and $r(p)_{K _{H}} \ge 2$.
\par
{\rm (ii)} $I _{H}/K _{H}$ is immediate and $I/K$ is inertial
relative to $\hat v _{H}$ and $v _{H}$, respectively.
\end{lemm}

\begin{proof}
Lemma 3.5 (ii) and the inequality $r(p)_{K} \ge 2$ ensure that
$v(K)/pv(K)$ is of order $p$. As $H \neq pH$, this implies that $v
_{H}(K) = pv _{H}(K)$. It is therefore clear from (2.3) that if
$v(p) \in H$, i.e. char$(K _{H}) = 0$, then finite extensions of $K$
in $K(p)$ are inertial relative to $v _{H}$, which yields
$\mathcal{G}(K(p)/K) \cong \mathcal{G}(K _{H}(p)/K _{H})$ (cf.
\cite{JW}, page 135), whence $r(p)_{K} = r(p)_{K _{H}}$. Suppose now
that $v(p) \notin pH$, fix an element $\pi \in M _{v}(K)$ so that
$v(\pi ) \in H \setminus pH$, and denote by $J$ the root field in $K
_{\rm sep}$ of the polynomial
$f(X) = X ^{p} - X - \pi ^{-1}$ over $K$. It is easy to see that $J
\in I(K(p)/K)$, $[J\colon K] = p$, and $J/K$ is inertial relative to
$v _{H}$ and totally ramified relative to $v$. In particular, $J
\neq I$ and $p \in P(K _{H})$, so it follows from Lemma 3.6, applied
to $K$, $v _{H}$ and $p$, that $K _{H}$ is $p$-quasilocal and Br$(K)
_{p} \cong {\rm Br}(K _{H}) _{p} \neq \{0\}$. In view of \cite{J},
Proposition~4.4.8, this indicates that $r(p)_{K _{H}} = \infty $. We
show that $I/K$ is inertial relative to $v _{H}$. This has already
been established in the case where $v(p) \in H$, so we assume here
that $v(p) \notin H$. It is clearly sufficient to prove that $IJ/J$
is inertial relative to the prolongation $v ^{\prime }_{H}$ of $v
_{H}$ on $J$. This implies that each generator $\psi $ of
$\mathcal{G}(I/K)$ is uniquely extendable to a generator $\psi
^{\prime }$ of $\mathcal{G}(IJ/J)$. We show that $IJ/J$ is inertial
relative to $v ^{\prime }_{H}$ by proving the following statement:
\par
\medskip
(3.1) There exists $r \in O _{v}(I)$, such that $v _{I}(r - \psi
^{\prime } (r)) \le p ^{-1}v(\pi )$.
\par
\medskip\noindent
Fix a root $\xi $ of $f$ in $J$, put $\theta = \xi ^{-1}$, and
denote by $H ^{\prime }$ the sum of $H$ and the cyclic group
$\langle v _{J}(\xi )\rangle $. It is easily verified that $H
^{\prime } \in {\rm Is}_{v _{J}}(J)$, $v ^{\prime }_{H} = v _{J,H'}$
and $v _{H}(\eta _{H}(\pi )) = pv ^{\prime }_{H}(\eta _{H'}(\theta
))$. For convenience, we put $\kappa _{H} = \eta _{H}(\kappa )$ and
$\kappa ^{\prime }_{H'} = \eta _{H'}(\kappa ^{\prime })$, for each
$\kappa \in O _{v _{H}}(K)$, $\kappa ^{\prime } \in O _{v'_{H}}(J)$.
Observing that char$(K _{H}) = p$ and $\theta _{H'} = \pi _{H}\prod
_{u=1} ^{p-1} (\xi _{H'} + u)$, one obtains by direct calculations
that $\hat v ^{\prime }_{H}(\theta _{H'}) = p ^{-1}\hat v _{H}(\pi
_{H})$ and $\hat v ^{\prime }_{H}(\theta _{H'} - \tilde \zeta
(\theta _{H'})) = (2p ^{-1})\hat v _{H}(\pi _{H})$, for each
generator $\tilde \zeta $ of $\mathcal{G}(J _{H'}/K _{H})$. Thus it
turns out that $v _{J}(\theta ) = p ^{-1}v(\pi )$, $v _{J}(\theta -
\zeta (\theta )) = (2p ^{-1})v(\pi )$ and $v _{J}(1 - \zeta (\theta
)\theta ^{-1}) = p ^{-1}v(\pi )$, provided that $\zeta $ generates
$\mathcal{G}(J/K)$. Hence, by (2.3) and the choice of $\pi $, $v
_{J}(\zeta (\theta )\theta ^{-1} - 1) \notin pv(J)$. Note also that
the $p$-quasilocal property of $K$ is preserved by $J$ \cite{Ch5},
I, Theorem~4.1, which ensures that $\zeta (\lambda )\lambda ^{-1}
\in N(IJ/J)$, for each $\lambda \in (IJ) ^{\ast }$ (see \cite{Ch6},
Lemma~4.2). Take an element $\theta ^{\prime } \in IJ$ of norm $N
_{J} ^{IJ} (\theta ^{\prime }) = \zeta (\theta )\theta ^{-1}$ and
put $\lambda ^{\prime } = \theta ^{\prime } - 1$. We show that $v
_{IJ}(\theta ^{\prime } - \psi ^{\prime }(\theta ^{\prime })) \le p
^{-1}v(\pi )$. It follows from the Henselian property of $v$ and the
primality of $p$ that $v _{IJ}(\theta ^{\prime } - \psi ^{\prime
}(\theta ^{\prime })) = v _{IJ}(\psi ^{\prime u} (\theta ^{\prime })
- \psi ^{\prime u'}(\theta ^{\prime }))$, for $\psi ^{\prime u} \neq
\psi ^{\prime u'}$. Therefore, the equality $N _{J} ^{IJ}(\theta
^{\prime }) = \zeta (\theta )\theta ^{-1}$ implies that if $v
_{IJ}(\theta ^{\prime } - \psi ^{\prime } (\theta ^{\prime }))
> p ^{-1}v(\pi )$, then $v _{IJ}(\lambda ^{\prime p}) = v _{J}(\zeta
(\theta )\theta ^{-1} - 1)$. Since $IJ/J$ is immediate relative to
$v _{J}$, our conclusion requires that $p ^{-1}v(\pi ) \in pv(J)$
and $v(\pi ) \in pv(K)$, a contradiction proving (3.1) (and the fact
that $I/K$ is inertial relative to $v _{H}$).
\par
It remains to be seen that $I _{H}/K _{H}$ is immediate relative to
$\hat v _{H}$. Observing that $v(I) = v(K)$ and $v(K)/H$ is
torsion-free, one obtains that $\hat v _{H}(I _{H}) = \hat v _{H}(K
_{H})$. Since, by Lemma 3.3 (iii), $\widehat K$ is perfect, and by
(2.1) (i), it is isomorphic to the residue field of $(K _{H}, \hat v
_{H})$, this implies that $I _{H}/K _{H}$ is immediate or inertial.
Suppose for a moment that $I _{H}/K _{H}$ is inertial. Then $I
_{H}/K _{H}$ possesses a primitive element $\tilde \alpha \in O
_{\hat v _{H}}(I _{H})$, such that $\hat v _{H}(d(\tilde g)) = 0$,
where $\tilde g$ is the minimal (monic) polynomial of $\tilde \alpha
$ over $K _{H}$, and $d(\tilde g)$ is the discriminant of $\tilde
g$. The choice of $\tilde \alpha $ guarantees that $\tilde g(X) \in
O _{\hat v _{H}}(K _{H}) [X]$, whence $\tilde g$ is a reduction
modulo $M _{v}(K)$ of a monic polynomial $g(X) \in O _{v}(K) [X]$
(see (2.1)). Denote by $d(g)$ the discriminant of $g$. It is easily
obtained that $v(d(g)) = 0$ and the residue class of $d(g)$ in $K
_{H}$ equals $d(\tilde g)$. Observe also that, for each root $\tilde
\beta \in O _{\hat v _{H}}(I _{H})$ of $\tilde g$, there exists a
root $\beta \in O _{v}(I)$ of $g$, such that $\hat \beta = \tilde
\beta $. The obtained result leads to the conclusion that $I/K$ is
inertial. This contradicts our assumptions and thereby proves that
$I _{H}/K _{H}$ is immediate relative to $\hat v _{H}$, as claimed.
\end{proof}

\medskip
\begin{rema}
The assertion of Lemma 3.7 (ii) can be restated by saying that
$v(\sigma (\lambda ) - \lambda ) > 0$ whenever $\lambda \in O
_{v}(I)$, and there exists $\alpha _{H} \in O _{v}(I)$, $v(\sigma
(\alpha _{H}) - \alpha _{H}) \in H$, where $\sigma $ is a generator
of $\mathcal{G}(I/K)$.
\end{rema}

\medskip
Our objective now is to prove Theorem 3.1, under the extra
hypothesis that Is$_{v}(K) \neq \{0\}$ and $H \neq pH$, for each $H
\in {\rm Is}_{v}(K)$, $H \neq \{0\}$. Suppose first that Is$_{v}(K)$
does not contain a minimal element (with respect to inclusion), and
fix an arbitrary element $\beta \in \nabla _{0}(K)$. Then $v(\beta -
1) \notin H _{\beta }$, for some $H _{\beta } \in {\rm Is}_{v}(K)$,
so it follows from (2.2), (3.1) and Lemma 3.7 that $\beta \subset
N(I/K)$. It remains to be seen that $\nabla _{0}(K) \subseteq
N(I/K)$, provided that Is$_{v} ^{\prime }(K)$ contains a minimal
element $\Gamma \neq \{0\}$. Applying Lemma 3.7, one sees that it
suffices to consider the special case of $v(K) = \Gamma $. The
minimality of $\Gamma $ indicates that it is Archimedean, so the
inclusion $I _{0}(K) \subset N(I/K)$ can be proved by showing that
$\nabla _{\delta }(K) \subseteq N(I/K)$, for an arbitrary $\delta
\in \Gamma $, $\delta > 0$. Our main step in this direction is
contained in the following lemma.
\par
\medskip
\begin{lemm}
Assume that $(K, v)$, $p$ and $I$ satisfy the conditions of Lemma
3.5, $v(K)$ is Archimedean, and $L \in I(K(p)/K)$ is a field, such
that $[L\colon K] = p$ and $L \neq I$. Suppose further that $v(L)
\setminus pv(L)$ contains an element $\gamma > 0$ satisfying the
conditions $\gamma = v _{L} (\lambda ) = v _{L}(\tau (\lambda
)\lambda ^{-1} - 1) < p ^{-1}v(p)$, for some $\lambda \in O
_{v}(L)$, where $\tau $ is a generator of $\mathcal{G}(IL/L)$. Then
$\nabla _{\gamma '}(L) \subseteq N(IL/L)$ and $\nabla _{\gamma
''}(K) \subseteq N(IL/K)$, where $\gamma ^{\prime } = (2p - 2)\gamma
$ and $\gamma ^{\prime \prime } = [(p ^{2} - 1)(4p - 2)]\gamma $.
\end{lemm}

\begin{proof}
Fix an element $\theta ^{\prime } \in (IL) ^{\ast }$ so that $N _{L}
^{IL} (\theta ^{\prime }) = \tau(\lambda )\lambda ^{-1}$, and put
$\tilde \gamma = v _{IL}(\tau ^{\prime } (\theta ^{\prime }) -
\theta ^{\prime })$, for some $\tau ^{\prime } \in
\mathcal{G}(IL/L)$, $\tau \neq 1$. As in the proof of (3.1), one
obtains that if $\tilde \gamma > \gamma $, then $v _{IL} (\lambda -
1 - (\theta ^{\prime } - 1) ^{p}) \ge \tilde \gamma $, which implies
$\gamma = v _{L}(\lambda - 1) = v _{IL}(\theta ^{\prime } - 1)
^{p}$. Since $IL/L$ is immediate, this contradicts the assumption
that $\gamma \notin pv(L)$ and thereby proves that $\tilde \gamma
\le \gamma $. Applying next (2.2) to the minimal polynomial of
$\lambda ^{\prime }$ over $L$, one obtains that $\nabla _{\gamma
'}(L) \subseteq N(IL/L)$. Observe that $pv(K)$ is a dense subgroup
of $\mathbb R$ ($I/K$ is immediate, whence $v(K)$ is noncyclic, see
\cite{TY}, Proposition~2.2). Therefore, for each $\varepsilon > 0$,
one can find an element $\mu _{\varepsilon } \in K$ such that $(2p -
3)\gamma < v(\mu _{\varepsilon }) < (2p - 3)\gamma + \varepsilon $
and $v(\mu _{\varepsilon }) \in pv(K)$. Hence, by the choice of
$\lambda $, $\gamma ^{\prime } < v _{L} (\mu _{\varepsilon }\lambda
)) < \gamma ^{\prime } + \varepsilon $, $v _{L}(\mu _{\varepsilon
}\lambda ) \notin pv(L)$, and $(2p - 1)\gamma < v _{L}(\tau ^{j} (1
+ \mu _{\varepsilon }\lambda ) - 1 - \mu _{\varepsilon }\lambda ) <
(2p - 1)\gamma + \varepsilon $, $j = 1, \dots , p - 1$. As $\nabla
_{\gamma '} \subseteq N(IL/L)$, these calculations prove the
existence of an element $\mu ^{\prime } \in \nabla _{0}(IL)$ of norm
$N _{L} ^{IL} (\mu ^{\prime }) = 1 + \mu _{\varepsilon }\lambda $.
Let $f$ be the minimal polynomial of $\mu ^{\prime }$ over $K$. It
is easily seen that $f$ is of degree $p ^{2}$. Using the above
calculations, observing that the natural action of
$\mathcal{G}(IL/K)$ on $L ^{\ast }$ induces on $N(IL/L)$ a structure
of a $\mathbb Z[\mathcal{G}(L/K)]$-module, and arguing as in the
proof of the inclusion $\nabla _{\gamma '} \subseteq N(IL/L)$, one
obtains that $0 < v _{IL}(\mu ^{\prime } - \varphi (\mu ^{\prime }))
\le (2p - 1)\gamma + \varepsilon $ when $\varphi $ runs across
$\mathcal{G}(IL/K) \setminus \{1\}$. Since $\varepsilon $ can be
taken smaller than any fixed positive number, this enables one to
deduce from (2.2) that $\nabla _{\gamma ''} \subseteq N(IL/L)$, so
Lemma 3.9 is proved.
\end{proof}

\medskip
It is now easy to prove Theorem 3.1 in the remaining case where
$v(K) \le \mathbb R$. Take elements $\gamma \in v(K) \setminus
pv(K)$ and $\tilde \mu \in K$ so that $\gamma > 0$ and $v(\tilde \mu
) = \gamma $. We prove that if $\gamma $ is sufficiently small, then
the extension $L _{\tilde \mu } = L$ of $K$ in $K _{\rm sep}$
generated by a root of the polynomial $f _{\tilde \mu }(X) = X ^{p}
- X - \tilde \mu ^{-1}$ satisfies the following conditions:
\par
\medskip
(3.2) $L \subseteq K(p)$, $[L\colon K] = p$ and there exists $\theta
\in L$, such that $v _{L}(\theta ) = p ^{-1}v(\tilde \mu ) = p
^{-1}\gamma $, $v _{L}(\zeta (\theta ) - \theta ) = (2p ^{-1})\gamma
$ and $v _{L}(\zeta (\theta )\theta ^{-1} - 1) = p ^{-1}\gamma $.
\par
\medskip
We show that one can take as $\theta $ the inverse of some root of
$f _{\tilde \mu }$. If char$(K) = p$, this is obtained by direct
calculations (as in the proof of (3.1)). Suppose further that
char$(K) = 0$, take a primitive $p$-th root of unity $\varepsilon
\in K _{\rm sep}$, and put $m = [K(\varepsilon )\colon K]$. It is
well-known (cf. \cite{L}, Ch. VIII, Sect. 3) that $K(\varepsilon
)/K$ is cyclic and $m \mid (p - 1)$. Set $\mu = m\tilde \mu $, fix a
generator $\varphi $ of $\mathcal{G}(K(\varepsilon )/K)$, and let
$s$ and $l$ be positive integers, such that $\varphi (\varepsilon )
= \varepsilon ^{s}$ and $p \mid (sl - 1)$. Denote by $\Lambda
^{\prime }$ some extension of $K(\varepsilon )$ in $K _{\rm sep}$
obtained by adjunction of a $p$-th root of the element $\rho (\mu )
= \prod _{u=0} ^{m-1} \varphi ^{i} (1 + (\varepsilon - 1) ^{p}\mu
^{-1}) ^{l ^{i}}$. It is easily verified that $\varphi (\rho (\mu
))\rho (\mu ) ^{-s} \in K(\varepsilon ) ^{\ast p}$. Observing also
that $v ^{\prime }(\rho (\mu ) - 1 - m(\varepsilon - 1 ) ^{p}\mu
^{-1}) \ge (2p)v ^{\prime }(\varepsilon - 1) - 2\gamma $, and that
the polynomial $g(X) = (X + 1) ^{p} - \rho (\mu )$ is irreducible
over $K(\varepsilon )$), one concludes that $\rho (\mu ) \notin
K(\varepsilon ) ^{\ast p}$. Hence, by Albert's theorem (cf.
\cite{A1}, Ch. IX), $\Lambda ^{\prime } = \Lambda (\varepsilon )$,
for some $\Lambda \in I(K(p)/K)$ with $[\Lambda \colon K] = p$. Our
calculations also show that $\Lambda ^{\prime }/K(\varepsilon )$ is
totally ramified. Since $m \mid (p - 1)$, this proves that $\Lambda
/K$ is totally ramified as well. Note further that when $\gamma $ is
sufficiently small, $\Lambda ^{\prime }/K(\varepsilon )$ possesses a
primitive element which is a root of the polynomial $X ^{p} - \tilde
\mu ^{p-1}X - \tilde \mu ^{p-1}$. This is obtained by applying (2.2)
to the polynomial $\tilde h(X) = g((\varepsilon - 1) ^{-1}\tilde \mu
X) \in O _{v'}(K(\varepsilon ))$. Thus it becomes clear that
$\Lambda ^{\prime }/K(\varepsilon )$ has a primitive element $\xi $
satisfying $f _{\tilde \mu }(X)$. This implies that $[K(\xi )\colon
K] = p$, so it follows from the cyclicity of $\Lambda ^{\prime }/K$
and Galois theory that $K(\xi ) = \Lambda = L$. Using again (2.2),
one concludes that when $\gamma $ is sufficiently small, the element
$\theta = \xi ^{-1}$ satisfies the inequalities required by (3.2).
Since $p ^{-1}\gamma \notin pv(L)$, the obtained result and Lemma
3.9 prove Theorem 3.1.

\medskip
Let $(K, v)$ be a Henselian field, such that char$(\widehat K) = p >
0$, and let $I/K$ be a $\mathbb Z_{p}$-extension, such that
$\widehat I = \widehat K$. Denote by $I _{n}$ the extension of $K$
in $I$ of degree $p ^{n}$, and put $v _{n} = v _{I _{n}}$, for each
$n \in \mathbb Z$, $n \ge 0$. The uniqueness, up-to an equivalence,
of $v _{n}$ implies the following inclusion, for every index $n$:
\par
\medskip
(3.3) $\{\psi _{n}(u _{n})u _{n} ^{-1}\colon \ u _{n} \in I _{n}
^{\ast }, \psi _{n} \in \mathcal{G}(I _{n}/K)\} \subseteq \nabla
_{0}(I _{n})$.
\par
\medskip\noindent
We say that $I$ is a norm-inertial extension of $K$, if $\nabla
_{0}(K) \subseteq N(I _{n}/K)$, for each $n \in \mathbb N$. Suppose
that $H \neq pH$, for every $H \in {\rm Is}_{v} ^{\prime }(K)$, $H
\neq \{0\}$. We conclude this Section with the proof of the
equivalence of the following statements in case $I/K$ is immediate:
\par
\medskip
(3.4) (i) $I/K$ is norm-inertial;
\par
(ii) $I/I _{n}$ is norm-inertial, for every index $n$;
\par
(iii) For each $\gamma \in v(K)$, $\gamma > 0$, there exists $\mu
_{n}(\gamma ) \in O _{v}(I _{n})$, such that $v _{n}(\varphi
_{n}(\mu _{n}(\gamma )) - \mu _{n}(\gamma )) < \gamma $, for each
$\varphi _{n} \in \mathcal{G}(I _{n}/K) \setminus \{1\}$.
\par
\medskip
The implication (3.4)(ii)$\to $(3.4) (i) is obvious and the
implication (3.4) (iii) $\to $(3.4) (ii) follows from (2.2),
\cite{Ch2}, II, (2.6) and (2.7), and the fact that $H _{n} \neq pH
_{n}$, for every $H _{n} \in {\rm Is}_{v _{n}} ^{\prime }(I _{n})$,
and each $n \in \mathbb N$. The implication (3.4) (i)$\to $(3.4)
(iii) can be deduced from the following result:
\par
\medskip
(3.5) Let $(K, v)$ be a Henselian field, $L \in I(K(p)/K)$ a cyclic
extension of $K$ of degree $p ^{n}$, $\psi $ a generator of
$\mathcal{G}(L/K)$, $\lambda $ and $\lambda _{0}$ be elements of $M
_{v}(L)$ and $M _{v}(K)$, respectively, such that $v(\lambda _{0})
\in v(K) \setminus pv(K)$ and $N _{K}^{L}(1 + \lambda ) = 1 +
\lambda _{0}$, where $N _{K}^{L}$ is the norm map. Then $v _{L}(\psi
^{j}(\lambda ) - \lambda ) \le v(\lambda _{0})$, for $j = 1, \dots ,
p ^{n} - 1$.
\par
\medskip
It is easy to see that $v _{L}(\psi (\alpha ) - \alpha ) \le v
_{L}(\psi ^{j}(\alpha ) - \alpha )$, for any $\alpha \in L$ and each
index $j$, and equality holds in the case where $p \dagger j$. When
$p \dagger j$, i.e. $\psi ^{j}$ generates $\mathcal{G}(L/K)$, this
leads to the conclusion of (3.5). Thus our assertion is proved in
case $n = 1$, so we assume further that $n \ge 2$. Let $k$ be an
integer with $1 \le k < n$, $L _{k}$ the fixed field of $\psi ^{p
^{k}} $, and $\lambda _{k} = -1 + \prod _{u=0} ^{p ^{k}-1} (1 + \psi
^{u}(\lambda ))$. Clearly, $N _{L _{k}}^{L}(1 + \lambda _{k}) = 1 +
\lambda _{0}$. Note also that $v _{L}(\psi ^{p ^{k}}(\lambda ) -
\lambda ) \le v(\lambda _{0})$, provided $v _{L}(\psi ^{p
^{k}}(\lambda _{k}) - \lambda _{k}) \le v(\lambda _{0})$. Since
$\psi ^{p ^{k}}$ is a generator of $\mathcal{G}(L/L _{k})$, these
observations enable one to complete the proof of (3.5) by a standard
inductive argument.

\medskip
At the same time, it is easily deduced from (2.2) (without
restrictions on $v(K)$) that the fulfillment of (3.4) (iii) ensures
that $I/K$ is norm-inertial.

\section{Proof of Theorem 1.1}

Let first $K$ be an arbitrary $p$-quasilocal nonreal field
containing a primitive $p$-th root of unity unless char$(K) = p$.
Then cd$(\mathcal{G}(K(p)/K)) \le 2$ and equality holds if and only
if char$(K) \neq p$ and Br$(K) _{p} \neq \{0\}$ (see \cite{Ch6},
Proposition~5.1, and \cite{S1}, Ch. I, 4.2). When $K$ possesses a
Henselian valuation $v$ with $v(K) \neq pv(K)$, this enables one to
deduce from \cite{Ch1}, (1.2), \cite{Ch5}, Lemma~3.6, and
\cite{Ch3}, Lemma~1.1 (b) (an analogue to a part of the main result
of \cite{MT}) that cd$(\mathcal{G}(K(p)/K)) = r(p)_{K}$. At the same
time, the assumptions on $K$, Lemma 3.3 (i) and \cite{Ch5}, I,
Lemma~3.5, indicate that $K$ is a nonreal field. These observations,
combined with \cite{W}, Lemma~7 (or \cite{Ch6}, Corollary~5.3),
prove (1.4). Using (1.4), Remark 3.4 and Lemma 3.3 (i), one deduces
the assertion of Theorem 1.1 in the special case where
char$(\widehat K) \neq p$.
\par
In the rest of our proof of Theorem 1.1, we assume that $(K, v)$ is
Henselian $p$-quasilocal with $v(K) \neq pv(K)$ and char$(\widehat
K) = p$. Suppose first that char$(K) = 0$ and $v _{G(K)}(K) \neq pv
_{G(K)}(K)$, $G(K)$ being defined as in Section 2, and fix a
primitive $p$-th root of unity $\varepsilon \in K _{\rm sep}$. Then
char$(K _{G(K)}) = 0$, so it follows from Remark 3.4 that
$\mathcal{G}(K _{G(K)}(p)/K _{G(K)}) \cong \mathbb Z_{p}$ unless $p
\notin P(K _{G(K)})$. In view of Remark 2.8 and Proposition 2.6,
this yields $G(K) = pG(K)$. Hence, by Proposition 2.5, $K
_{G(K)}(p)/K _{G(K)}$ is immediate relative to $\hat v _{G(K)}$
unless $p \in P(\widehat K)$, and by (1.4) and Remark 3.4, applied
to $(K, v _{G(K)})$, $\mathcal{G}(K(p)/K)$ has the following
properties:
\par
\medskip
(4.1) (i) $\mathcal{G}(K(p)/K) \cong \mathcal{G}(K _{G(K)}(p)/K
_{G(K)})$, provided that $\varepsilon \notin K$; when this occurs,
$r(p)_{K} \le 1$ and Br$(K) _{p}$ is isomorphic to Br$(K _{G(K)})
_{p}$ or to a divisible hull of $_{p}{\rm Br}(K _{G(K)}) \oplus
v(K)/pv(K)$, depending on whether or not $r(p)_{K} = 0$;
\par
(ii) If $\varepsilon \in K$ and either $p \in P(K _{G(K)})$ or
$v(K)/pv(K)$ is noncyclic, then $r(p)_{K} = 2$,
$\mathcal{G}(K(p)/K)$ is a Demushkin group and Br$(K) _{p} \cong
\mathbb Z(p ^{\infty })$;
\par
(iii) $\mathcal{G}(K(p)/K) \cong \mathbb Z_{p}$, if $\varepsilon \in K$,
$p \notin P(K _{G(K)})$ and $v(K)/pv(K)$ is of order $p$; in this
case, Br$(K) _{p} = \{0\}$ and finite extensions of $K$ in $K(p)$
are totally ramified.
\par
\medskip\noindent
When $p \in P(K _{G(K)})$, it also becomes clear that the compositum
$K _{G(K)} ^{\prime }$ of the inertial lifts in $K _{\rm sep}$,
relative to $v _{G(K)}$, of the finite extensions of $K _{G(K)}$ in
$K _{G(K)}(p)$, has the following properties:
\par
\medskip
(4.2) $K _{G(K)} ^{\prime }$ is a $\mathbb Z _{p}$-extension of $K$
with $v(K _{G(K)} ^{\prime }) = v(K)$; more precisely $K _{G(K)}
^{\prime }/K$ is immediate relative to $v$ unless $p \in P(\widehat
K)$.
\par
\medskip\noindent
Statements (4.1), (4.2) and Corollary 2.7 reduce the proof of
Theorem 1.1 to the special case where char$(K) = p$ or char$(K) = 0$
and the group $v(K)/G(K)$ is $p$-divisible. Then it follows from
Corollary 2.7 and Remark 2.8 that $r(p)_{K} \ge 2$ and $v(K)/pv(K)$
is of order $p$. This, combined with (2.3), (2.4) (i) and Lemma 3.3,
proves the following assertions:
\par
\medskip
(4.3) If $d(K)$ contains a noncommutative defectless $K$-algebra of
$p$-primary index, then the compositum $U _{p}(K)$ of the inertial
extensions of $K$ in $K(p)$ is a $\mathbb Z_{p}$-extension of $K$.
In addition, every $D \in d(K)$ and each finite extension of $K$ in
$K(p)$ are defectless over $K$.
\par
\medskip\noindent
Note also that $\nabla _{0}(K) \subseteq N(U/K)$, for every inertial
extension $U/K$; this is a well-known consequence of (2.2). When
$U/K$ is cyclic and $U \subseteq K(p)$, this enables one to deduce
from \cite{P}, Sect. 15.1, Proposition~b, that Br$(U/K) = \{b \in
{\rm Br}(K)\colon \ [U\colon K]b = 0\}$ and Br$(U/K)$ is cyclic of
order $[U\colon K]$. Thereby, it becomes clear that Br$(U _{p}(K)/K)
= {\rm Br}(K) _{p} \cong \mathbb Z(p ^{\infty })$, which proves the
assertion of Theorem 1.1 in the case singled out by (4.3).
\par
Suppose now that $K$ has a finite extension $L ^{\prime }$ in $K(p)$
of nontrivial defect, choose $L ^{\prime }$ to be of minimal
possible degree over $K$, and fix a maximal subfield $L$ of $L
^{\prime }$ including $K$. Clearly, $L ^{\prime }/L$ is immediate
and $[L ^{\prime }\colon L] = p$, and by \cite{Ch5}, Theorem~4.1,
$L$ is $p$-quasilocal. In addition, it follows from (1.1) and Lemma
3.3 (i) that Cor$_{L/K}$ induces an isomorphism of Br$(L) _{p}$ on
Br$(K) _{p}$. Observing also that $v(K)/pv(K) \cong v(L)/pv(L)$
(see, e.g., \cite{Ch1}, Remark~2.2), one concludes that $L$, $v
_{L}$ and $p$ satisfy the conditions of Theorem 3.1, which yields
Br$(K) _{p} \cong \mathbb Z(p ^{\infty })$, as claimed. For the rest of
the proof of Theorem 1.1, we need the following lemma.

\par
\medskip
\begin{lemm} In the setting of Theorem 1.1, suppose that {\rm
char}$(\widehat K) = p$ and there exists a finite extension of $K$
in $K(p)$ of nontrivial defect. Then there is a field $I \in
I(K(p)/K)$, such that $I/K$ is immediate and $[I\colon K] = p$.
\end{lemm}

\begin{proof}
In view of Galois theory and the subnormality of proper subgroups of
finite groups, it suffices to consider the special case in which $K$
has an extension $M$ in $K(p)$ of degree $p ^{2}$ and defect $p$. We
show that there exists a field $I \in I(M/K)$, such that $[I\colon
K] = p$ and $I/K$ is immediate. Let $R$ be an extension of $K$ in
$M$ of degree $p$. It follows from (2.3) that $v(R)/pv(R) \cong
v(K)/pv(K)$, so we have $v(R) \neq pv(R)$. If $R/K$ is immediate,
there is nothing to prove, so we assume that this is not the case.
Our extra hypothesis guarantees that $M/R$ is immediate, and since
$R$ is $p$-quasilocal \cite{Ch5}, I, Theorem~4.1 (i), one obtains
from Lemma 3.5 (i) and (ii) that $v(R)/pv(R)$ and $v(K)/pv(K)$ are
of order $p$, $\widehat R$ is perfect and $p \notin P(\widehat R)$.
Note further that, by Lemma 3.3 (ii), $M$ is the only extension of
$R$ in $R(p) = K(p)$ of degree $p$, which is not totally ramified.
Statement (2.3) and these observations indicate that $p \notin
P(\widehat K)$ and $R/K$ is totally ramified. At the same time, by
Theorem 3.1, Br$(R) _{p} \cong \mathbb Z(p ^{\infty })$, so it
follows from (1.1) and Lemma 3.3 (i) that Br$(K) _{p} \cong \mathbb
Z(p ^{\infty })$. Using the normality of $R/K$ and the equality
$[M\colon K] = p ^{2}$, one proves that $M/K$ is abelian. The
obtained results, combined with (1.2) (ii) and \cite{Ch6},
Lemma~2.1, imply that $\mathcal{G}(M/K) \cong K ^{\ast }/N(M/K)$. We
show that $M/K$ is noncyclic. It is clear from Theorem 3.1 and the
Henselian property of $v$ that $N(M/R) = \{\rho \in R ^{\ast }\colon
\ v _{R} (\rho ) \in pv(R)\}$. This, combined with Hilbert's Theorem
90 and the transitivity of norm maps in the field tower $K \subset R
\subset M$, implies that $N(M/K)$ is included in the set $\Omega
_{p}(K) = \{\alpha \in K ^{\ast }\colon \ v(\alpha ) \in pv(K)\}$.
Thus it turns out that $K ^{\ast p} \subseteq N(M/K)$, whence $K
^{\ast }/N(M/K)$ has exponent $p$. As $\mathcal{G}(M/K) \cong K
^{\ast }/N(M/K)$, it is now easy to see that $\mathcal{G}(M/K)$ is
noncyclic. By Galois theory, this means that $M = RL$, for some $L
\in I(M/K)$ with $[L\colon K] = p$ and $L \neq R$. Clearly, one may
assume for the rest of the proof that $L/K$ is totally ramified. By
(1.2), $K ^{\ast } = N(R/K)N(L/K)$, $K ^{\ast }/N(R/K)$ and $K
^{\ast }/N(L/K)$ are of order $p$, and $N(M/K) = N(R/K) \cap
N(L/K)$. As $v(K)/pv(K)$ is of order $p$, these observations prove
the existence of elements $\lambda $ and $r$ of $K ^{\ast }$, such
that $v(\lambda ) = v(r) \notin pv(K)$, $\lambda \in N(L/K)$, $r \in
N(R/K)$ and the co-sets $\lambda N(M/K)$ and $rN(M/K)$ generate $K
^{\ast }/N(M/K)$. In addition, it follows from (1.2) that the set
$\{N(K ^{\prime }/K)\colon \ K ^{\prime } \in I(M/K)\}$ equals the
set of subgroups of $K ^{\ast }$ including $N(M/K)$. Therefore,
there is $I \in I(M/K)$, such that $[I\colon K] = p$ and $N(I/K)$ is
generated by $N(M/K)$ and $\lambda r ^{-1}$. Hence, $N(I/K) = \Omega
_{p}(K)$, which means that $I/K$ is not totally ramified. As $M/R$
is immediate, this leads to the conclusion that $I/K$ is also
immediate, which proves Lemma 4.1.
\end{proof}

The idea of the remaining part of the proof of Theorem 1.1 is to
show that if finite extensions of $K$ in $K(p)$ are defectless, then
every $\Delta \in d(K)$ is defectless over $K$. Its implementation
relies on the following two lemmas.

\medskip
\begin{lemm}
Let $K$ be a $p$-quasilocal field, for some $p \in \mathbb P$. Assume
that {\rm Br}$(K) _{p}$ is divisible, $M$ is an extension of $K$ in
$K(p)$ of degree $p$, $\psi $ is a generator of $\mathcal{G}(M/K)$
and $d(K)$ contains the algebra $(M/K, \psi , c)$, for some $c \in K
^{\ast } \setminus M ^{\ast p}$. Then $[(M/K, \psi , c)] = {\rm
Cor}_{M/K}([(M _{1}/M, \psi _{1}, c)])$, for some $M _{1} \in
I(K(p)/M)$ with $[M _{1}\colon M] = p$, and some generator $\psi
_{1}$ of $\mathcal{G}(M _{1}/M)$.
\end{lemm}

\begin{proof}
Suppose first that char$(K) = p$. By the Artin-Schreier theorem (cf.
\cite{L}, Ch. VIII, Sect. 6), then $M = K(\xi )$, where $\xi $ is a
root of the polynomial $X ^{p} - X - a$, for some $a \in K ^{\ast
}$. Clearly, $a$ can be chosen so that $(M/K, \psi , c)$ is
isomorphic to the $p$-symbol $K$-algebra $K[a, c)$, in the sense of
\cite{TY}. Since $M/K$ is separable, there exists $\eta \in M$ of
trace Tr$_{K}^{M}(\eta ) = c$. This implies that the polynomial $X
^{p} - X - \eta $ has no zero in $M$, whence, by the Artin-Schreier
theorem, its root field over $M$ is a cyclic extension of $M$ of
degree $p$. Since, by a known projection formula (see \cite{MM},
Proposition~3 (i)), Cor$_{M/K}([M[\eta , c)]) = [K[a, c)]$, these
observations prove Lemma 4.2 in the case of char$(K) = p$.
\par
In the rest of the proof, we assume that char$(K) = 0$, $\varepsilon
$ is a primitive $p$-th root of unity in $K _{\rm sep}$,
$[K(\varepsilon )\colon K] = m$ and $s$ is an integer satisfying the
equality $\varphi (\varepsilon ) = \varepsilon ^{s}$, where $\varphi
$ is a generator of $\mathcal{G}(K(p)(\varepsilon )/K(p))$. Then it
follows from (1.1), Lemma 3.3 (i) and \cite{Ch5}, I, Theorems~3.1
(i) and 4.1 (iii), that ind$(A) = {\rm ind}(A ^{\prime })$ whenever
$A \in d(M)$, $A ^{\prime } \in d(K)$, $[A] \in {\rm Br}(M) _{p}$
and Cor$_{M/K}([A]) = [A ^{\prime }]$. When $\varepsilon \in K$,
$F/K$ is cyclic of degree $p$, by Kummer theory, and by the
assumption on $c$, $F \neq M$, which implies $MF/M$ is cyclic and
$[MF\colon M] = p$. Since $M$ is $p$-quasilocal, these observations
enable one to deduce the assertion of Lemma 4.2 from Lemma 3.3.
\par
Suppose now that $\varepsilon \notin K$ and, for each $R \in
I(K(p)/K)$, let $R _{\varepsilon } = \{r \in R(\varepsilon ) ^{\ast
}\colon \ \varphi (r)r ^{-s} \in R(\varepsilon ) ^{\ast }\}$. It
follows from Albert's theorem (see \cite{A1}, Ch. IX, Theorem~6)
that $M(\varepsilon )$ is generated over $K(\varepsilon )$ by a
$p$-th root of some element $\mu \in K _{\varepsilon }$. In
addition, it becomes clear that $\mu $ can be chosen so that the
symbol $K(\varepsilon )$-algebra $A _{\varepsilon }(\mu , c;
K(\varepsilon ))$ is isomorphic to $(M/K, \psi , c) \otimes _{K}
K(\varepsilon )$. As $p > 2$, one also sees that $\mu \in
N(M(\varepsilon )/K(\varepsilon ))$, whence $\mu $ equals the norm
$N_{K(\varepsilon )}^{M(\varepsilon )}(\mu _{1}\kappa )$, for some
$\mu _{1} \in M _{\varepsilon }$, $\kappa \in K(\varepsilon )$ (see
\cite{Ch8}, Lemma~3.1). Denote by $M _{1} ^{\prime }$ the extension
of $M(\varepsilon )$ obtained by adjunction of the $p$-th roots of
$\mu _{1}$ in $K _{\rm sep}$. Since $\mu _{1} \notin M(\varepsilon )
^{\ast p}$, Albert's theorem indicates that $M _{1} ^{\prime } = M
_{1}(\varepsilon )$, for some $M _{1} \in I(K(p)/M)$ with $[M
_{1}\colon M] = p$. Thereby, it becomes clear that the symbol
$M(\varepsilon )$-algebra $A _{\varepsilon }(\mu _{1}, c;
M(\varepsilon ))$ is isomorphic to $(M _{1}/M, \psi _{1}, c) \otimes
_{M} M(\varepsilon )$, for some generator $\psi _{1}$ of
$\mathcal{G}(M _{1}/M)$. Applying the projection formula for symbol
algebras (cf., e.g., \cite{Ti}, Theorem~3.2), one concludes that
Cor$_{M(\varepsilon )/K(\varepsilon )}([A _{\varepsilon }(\mu _{1},
c; M(\varepsilon ))]) = [A _{\varepsilon }(\mu , c; K(\varepsilon
))]$. Using also the $K(\varepsilon )$-isomorphism $A _{\varepsilon
}(\mu , c; K(\varepsilon )) \cong (M/K, \psi , c)$ $\otimes _{K}
K(\varepsilon )$ (and the fact that $m \mid (p - 1)$), one obtains
from the RC-formula that Cor$_{M/K}([M _{1}/M, \psi _{1}, c)]) =
[(M/K, \psi , c)]$, as claimed by Lemma 4.2.
\end{proof}

\medskip
\begin{lemm}
Let $(K, v)$ be a Henselian $p$-quasilocal field with {\rm char}$(K)
= 0$, {\rm char}$(\widehat K) = p$ and $v(K) \neq pv(K)$. Assume
that finite extensions of $K$ in $K(p)$ are totally ramified, $(M,
c)$ is a pair satisfying the conditions of Lemma 4.2, and $F$ is an
extension of $K$ in $K _{\rm sep}$ obtained by adjunction of a
$p$-th root of $c$. Then the extension $MF/M$ is totally ramified of
degree $p$.
\end{lemm}

\begin{proof}
We retain notation as in the proof of Lemma 4.2. In view of Kummer
theory, there is nothing to prove in case $\varepsilon \in K$, so we
assume that $\varepsilon \notin K$. Then it follows from Lemma 3.3
(iii) that $R _{\varepsilon } \subseteq \nabla _{0}(R(\varepsilon
))R(\varepsilon ) ^{\ast p}$, for each $R \in I(K(p)/K)$. Applying
Albert's theorem, Lemma 3.3 and \cite{Ti}, Theorems~2.5 and 3.2, as
in the proof of Lemma 4.2, one obtains the following result:
\par
\medskip
(4.4) The $M(\varepsilon )$-algebra $A _{\varepsilon }(r, c;
M(\varepsilon ))$ lies in $d(M(\varepsilon ))$ if and only if
\par\noindent
$A _{\varepsilon }(N _{\varepsilon }^{M(\varepsilon )}(r), c;
K(\varepsilon )) \in d(K(\varepsilon ))$; equivalently, $r \in
N(FM(\varepsilon )/M(\varepsilon ))$ if and only if
$N_{K(\varepsilon )}^{M(\varepsilon )}(r) \in N(F(\varepsilon
)/K(\varepsilon ))$.
\par
\medskip\noindent
We prove that $MF/M$ is totally ramified by assuming the opposite.
In view of (2.3) and Lemma 3.3 (iii), this requires that $MF/M$ is
inertial or immediate, and since $m \mid (p - 1)$, $MF(\varepsilon
)/M(\varepsilon )$ must be subject to the same alternative. It
follows from the Henselity of $v _{K(\varepsilon )}$ that if
$F(\varepsilon )/K(\varepsilon )$ is inertial, then $I
_{0}(K(\varepsilon )) \subseteq N(F(\varepsilon )/K(\varepsilon ))$.
Applying now \cite{P}, Sect. 15.1, Proposition~b, one concludes that
$(M _{1}/M, \psi _{1} , c) \notin d(K)$. Since $(M/K, \psi , c) \in
d(K)$, this contradicts Lemma 4.2, and thereby proves that
$MF(\varepsilon )/M(\varepsilon )$ is not inertial. The final step
towards the proof of Lemma 4.3 relies on the fact that $M
_{\varepsilon }$ is a module over the integral group ring $\mathbb Z
[\mathcal{G}(M(\varepsilon )/K(\varepsilon ))]$. In view of (4.4),
this ensures that $\theta (r)r ^{-1} \in N(MF(\varepsilon
)/M(\varepsilon ))$ whenever $\theta \in \mathcal{G}(M(\varepsilon
)/K(\varepsilon ))$ and $r \in M _{\varepsilon }$. Assuming now that
$MF/M$ is immediate (or equivalently, that $F/K$ is immediate),
using the condition on the finite extensions of $K$ in $K(p)$, and
arguing as in the proof of (3.1) and Lemma 3.9, one obtains from
this result that $\nabla _{0}(K(\varepsilon )) \subseteq
N(F(\varepsilon )/K(\varepsilon ))$. This, however, contradicts the
fact that $(M/K, \psi , c) \in d(K)$, so Lemma 4.3 is proved.
\end{proof}

\medskip
\begin{rema}
Assume that $(K, v)$ is a Henselian field, such that char$(K) = p$,
$v(K)/pv(K)$ is of order $p$ and finite extensions of $K$ in $K(p)$
are totally ramified. Then, by the proofs of \cite{TY}, Lemmas~2.2,
3.2 and 3.3, the assertion of Lemma 4.3 remains valid without the
assumption that $K$ is $p$-quasilocal (when $K _{\rm sep}$ is
replaced by its algebraic closure). Hence, by the proof of
\cite{TY}, Proposition~3.2, $d(K)$ does not contain cyclic
$K$-algebras of index $p$. This, combined with \cite{A2}, Ch. VII,
Theorem~28, implies Br$(K) _{p} = \{0\}$.
\end{rema}

\medskip
Let now $(K, v)$ be a Henselian $p$-quasilocal with char$(K) = 0$
and char$(\widehat K) = p$, and suppose that $p \in P(K)$,
$v(K)/pv(K)$ is of order $p$ and finite extensions of $K$ in $K(p)$
are defectless. As noted at the beginning of this Section, then
$r(p)_{K} \ge 2$. Assume further that Br$(K) _{p} \neq \{0\}$ and
fix a primitive $p$-th root of unity $\varepsilon \in K _{\rm sep}$.
Since $K$ is $p$-quasilocal, this implies the existence of a cyclic
algebra $D \in d(K)$ of index $p$. As in the proofs of Lemmas 4.2
and 4.3, it is seen that $D \otimes _{K} K(\varepsilon )$ is
$K(\varepsilon )$-isomorphic to $A _{\varepsilon }(a, b;
K(\varepsilon ))$, for some $a \in K ^{\ast }$, $b \in K
_{\varepsilon }$. In addition, it turns out that, by the proof of
Albert's cyclicity criterion for an algebra $\Delta \in d(K)$ of
index $p$ (see \cite{P}, Sect. 15.5, or \cite{Ch8}, (3.3)) that if
$A _{\varepsilon }(a, b; K(\varepsilon )) \cong A _{\varepsilon }(a;
b _{0}; K(\varepsilon ))$, for some $b _{0} \in \nabla
_{0}(K(\varepsilon ))$, then there exists $b _{0} ^{\prime } \in K
_{\varepsilon }$, such that $A _{\varepsilon }(a, b _{0} ^{\prime };
K(\varepsilon )) \cong A _{\varepsilon }(a, b; K(\varepsilon ))$ and
$v _{K(\varepsilon )}(b _{0} ^{\prime } - 1) \ge v _{K(\varepsilon
)}(b _{0} - 1)$. Considering now $A _{\varepsilon }(a, b;
K(\varepsilon ))$ as in the proof of \cite{TY}, Proposition~3.3 (see
also \cite{Schil}, Ch. 2, Lemma~19), and using Lemma 4.3 and the
inequality $r(p)_{K} \ge 2$, one obtains the following result:
\par
\medskip
(4.5) The $K$-algebra $D$ is defectless.
\par
\medskip\noindent

\par
\medskip
We are now in a position to complete the proof of Theorem 1.1.
Statements (4.1), (1.4) and (4.5), combined with Lemma 3.3 (ii) and
Remark 3.4, as well as with Remark 4.4 and the observation preceding
the statement of Lemma 4.1, prove Theorem 1.1 (iii). The conclusion
of Theorem 1.1 (i) follows from Lemma 3.3, Remark 3.4, the comment
on (4.3) and the pointed observation. It remains for us to prove
Theorem 1.1 (ii). Suppose that $(K, v)$ satisfies the conditions of
Lemma 4.1. Then $K$ has an immediate extension $I _{1}$ in $K(p)$ of
degree $p$. At the same time, Lemma 3.3 (ii), combined with (2.4)
(i), (4.3), (4.5) and Remark 4.4, indicates that $(R, v _{R})$
satisfies the conditions of Lemma 4.1 whenever $R \in I(K(p)/K)$ and
$[R\colon K] \in \mathbb N$. Therefore, one proves without
difficulty by induction on $n$ the existence of a unique sequence $I
_{n}$, $n \in \mathbb N$, of subfields of $K _{\rm sep}$, such that
$I _{1} = I$, $I _{n} \subset I _{n+1}$, $[I _{n+1}\colon I _{n}] =
p$, and $I _{n+1}/I _{n}$ is immediate, for every index $n$. In view
of Galois theory, this implies that $I _{n}/K$ is cyclic and
immediate with $[I _{n}\colon K] = p ^{n}$ whenever $n \in \mathbb
N$, and the union $I _{\infty } = \cup _{n=1} ^{\infty } I _{n}$ is
the unique immediate $\mathbb Z_{p}$-extension of $K$ in $K _{\rm
sep}$. This, combined with (4.1), (4.3) and Theorem 1.1 (iii),
yields the alternative of Theorem 1.1 (ii) in the case of
char$(\widehat K) = p$. Thus Theorem 1.1 is proved.

\begin{rema}
The proof of Theorem 1.1 is considerably easier in the special case
where $K _{\rm sep} = K(p)$, since then $K$ contains a primitive
$p$-th root of unity or char$(K) = p$, which simplifies the
consideration of the structure of Br$(K) _{p}$ (see (4.1) and Remark
3.4). In addition, when char$(\widehat K) = p$, (4.5) can be
directly deduced from \cite{TY}, Theorem~3.1.
\end{rema}

\section{Brauer groups of Henselian PQL-fields with totally
indivisible value groups}

The purpose of this Section is to describe the isomorphism classes
of several major types of valued PQL-fields considered in this
paper. Our first result is particularly useful in the case of
quasilocal fields:

\medskip
\begin{prop}
Let $(K, v)$ be a Henselian field, such that $v(K) \neq pv(K)$ and
{\rm Br}$(K) _{p} \neq \{0\}$, for some $p \in \mathbb P$. Suppose
further that finite extensions of $K$ are $p$-quasilocal. Then $p
\in P(K)$ and {\rm Br}$(K) _{p} \cong \mathbb Z(p ^{\infty })$.
Moreover, every $D \in d(K)$ of $p$-primary index is a cyclic
$K$-algebra.
\end{prop}

\begin{proof} Fix Sylow pro-$p$-subgroups $G _{p}$ and $\widetilde G _{p}$
of $\mathcal{G}_{K}$ and $\mathcal{G}_{\widehat K}$, respectively,
and denote by $K _{p}$ and $\widehat K _{p}$ the corresponding fixed
fields. Our choice of $K _{p}$ ensures that $p \dagger [K ^{\prime
}\colon K]$ whenever $K ^{\prime } \in I(K _{p}/K)$ and $[K ^{\prime
} \colon K] \in \mathbb N$, so \cite{Ch1}, (1.2) and Remark~2.2, and the
results of \cite{P}, Sect. 13.4, imply the following:
\par
\medskip
(5.1) Br$(K _{p}/K) \cap {\rm Br}(K) _{p}$ $= \{0\}$ and the natural
embedding of $K$ into $K _{p}$ induces a group isomorphism
$v(K)/pv(K) \cong v(K _{p})/pv(K _{p})$.
\par
\medskip\noindent
Since Br$(K) _{p} \neq \{0\}$ and, by \cite{Ch5}, I, Lemma~8.3, $K
_{p}$ is $p$-quasilocal, it can be easily deduced from (5.1), Lemma
3.3 (i) and Theorem 1.1 that Br$(K _{p}) \cong \mathbb Z(p ^{\infty
}) \cong {\rm Br}(K)$. In view of \cite{Ch5}, I, Theorem~3.1 (i), it
remains to be seen that $p \in P(K)$. When $p = {\rm char}(K)$, this
follows from Lemma 2.2, so we assume further that $p \neq {\rm
char}(K)$. As Br$(K) _{p} \neq \{0\}$, Theorem~2 of \cite{M}, Sect.
4, requires the existence of an algebra $D \in d(K)$ of index $p$.
Fix an element $\theta \in K ^{\ast }$ so that $v(\theta ) \notin
pv(K)$ and denote by $L _{\theta }$ some extension of $K$ in $K
_{\rm sep}$ generated by a $p$-th root of $\theta $. Also, let
$\varepsilon \in K _{\rm sep}$ be a primitive $p$-th root of unity.
Since $[K(\varepsilon )\colon K] \mid (p - 1)$, we have $D \otimes
_{K} K(\varepsilon ) \in d(K(\varepsilon ))$ and $[L _{\theta
}\colon K] = [L _{\theta }(\varepsilon )\colon K(\varepsilon )] =
p$. Observing also that $L _{\theta }(\varepsilon )$ is cyclic over
$K(\varepsilon )$, one obtains from the $p$-quasilocal property of
$K(\varepsilon )$ that $L _{\theta }(\varepsilon )$ embeds in $D
\otimes _{K} K(\varepsilon )$ as a $K(\varepsilon )$-subalgebra.
Thus it becomes clear that $L _{\theta }$ is isomorphic to a
$K$-subalgebra of $D$, so it follows from Albert's criterion (cf.
\cite{P}, Sect. 15.5) that $D$ is a cyclic $K$-algebra. This shows
that $p \in P(K)$, so Proposition 5.1 is proved.
\end{proof}

\begin{rema}
In the setting of Proposition 5.1, let $v(K)/pv(K)$ be of order $p$,
and for each finite extension $L$ of $K$ in $K _{\rm sep}$, put $X
_{p}(L) = \{\chi \in C_{L}\colon \ p\chi = 0\}$, where $C _{L}$ is
the continuous character group of $\mathcal{G}_{L}$. Denote by
$\kappa _{L}$ the canonical pairing $X _{p}(L) \times v(K)/pv(K) \to
_{p}{\rm Br}(L)$ (see, e.g., the proof of \cite{Ch5}, I, Lemma~1.1),
and by cor$_{L/K}$ the homomorphism of $X _{p}(L)$ into $X _{p}(K)$
induced by the corestriction map $C_{L} \to C_{K}$. It is clear from
Proposition 5.1 that $\rho _{K/L}$ maps Br$(K) _{p}$ surjectively
upon Br$(L) _{p}$, and by a well-known projection formula (stated in
\cite{S2}, page 205, for a proof, see, e.g., \cite{We},
Proposition~4.3.7), the compositions $\kappa _{K} \circ {\rm
cor}_{L/K}$ and Cor$_{L/K} \circ \kappa _{L}$ coincide. Choose as we
can (by Remark 4.5) $L$ so that $p \dagger [L\colon K]$ and there is
a field $I _{L} \in I(L(p)/L)$, such that $[I _{L}\colon L] = p$ and
$v(I _{L}) = v(L)$. Then it follows from \cite{Ti}, Theorem~3.2, and
the noted property of $\rho _{K/L}$ that Cor$_{L/K}$ induces an
isomorphism Br$(L) _{p} \cong {\rm Br}(K) _{p}$. Fix a generator
$\sigma _{L}$ of $\mathcal{G}(I _{L}/L)$ and an element $\pi \in K
^{\ast }$ so that $v(\pi ) \notin pv(K)$. The preceding observations
show that Cor$_{L/K}([I _{L}/L, \sigma _{L}, \pi )]) \in {\rm Br}(I
_{K}/K)$, for some $I _{K} \in I(K(p)/K)$ with $[I _{K}\colon K] =
p$ and $v(I _{K}) = v(K)$. Hence, by Lemma 3.3 (ii), $I _{L} = I
_{K}L$, which yields $p \in P(K)$ independently of the proof of
(4.5).
\end{rema}

It is known (and easy to see, e.g., from Scharlau's generalization
of Witt's decomposition theorem or from \cite{Ch5}, I,
Corollary~8.5) that every divisible subgroup $T$ of $\mathbb
Q/\mathbb Z$ is isomorphic to Br$(K _{T})$, provided that $(K, v)$
is a Henselian discrete valued field with a quasifinite residue
field and $K _{T}$ is the compositum of the inertial extensions of
$K$ in $K _{\rm sep}$ of degrees not divisible by any $p \in \mathbb
P$, for which $T _{p} \neq \{0\}$. This, combined with our next
result, describes the isomorphism classes of Brauer groups of
Henselian quasilocal fields with totally indivisible value groups.
\par
\medskip
\begin{coro} Let $(K, v)$ be a Henselian quasilocal field,
such that $v(K)$ is totally indivisible. Then $K$ is nonreal,
$\widehat K$ is perfect, every $D \in d(K)$ is cyclic and {\rm
Br}$(K)$ is divisible and embeddable in $\mathbb Q /\mathbb Z$. Moreover,
$P(K)$ contains every $p \in \mathbb P$, for which {\rm Br}$(K) _{p} \neq
\{0\}$.
\end{coro}

\begin{proof}
Our concluding assertion and the one concerning Br$(K)$ follow from
Proposition 5.1. The statements that $\widehat K$ is perfect and $K$
is nonreal are implied by Lemma 3.3, the assumption on $v(K)$ and
\cite{Ch5}, I, Lemma~3.6. Note finally that all $D \in d(K)$ are
cyclic. Since it suffices to prove this only in the special case
where $[D] \in {\rm Br}(K) _{p}$, for some $p \in \mathbb P$ (see
\cite{P}, Sect. 15.3), the assertion can be viewed as a consequence
of Proposition 5.1.
\end{proof}

\smallskip
\begin{coro}
Under the hypotheses of Corollary 5.3, let $L/K$ be a finite abelian
extension and $L _{0}$ the maximal extension of $K$ in $L$, for
which $[L _{0}\colon K]$ is not divisible by any $p \in \mathbb P$
with Br$(K) _{p} = \{0\}$. Then $E ^{\ast }/N(L/K) \cong
\mathcal{G}(L _{0}/K)$.
\end{coro}

\begin{proof}
This follows at once from (1.2) (ii), Theorem 1.1 and \cite{Ch6},
Lemma~2.1.
\end{proof}

\medskip
\begin{coro}
An abelian torsion group $T$ with $T _{2} \neq \{0\}$ is isomorphic
to {\rm Br}$(K)$, for some Henselian {\rm PQL}-field $(K, v)$ such
that $v(K)$ is totally indivisible, if and only if $T$ is divisible
and $T _{2} \cong \mathbb Z(2 ^{\infty })$. When this holds, $T
\cong {\rm Br}(F)$, for some Henselian discrete valued {\rm
PQL}-field $(F, w)$.
\end{coro}

\begin{proof} Theorem 3.1 shows that Br$(K) _{2} \cong \mathbb Z(2
^{\infty })$ whenever $(K, v)$ is a Henselian $2$-quasilocal field
with $v(K) \neq 2v(K)$ and Br$(K) _{2} \neq \{0\}$. This, combined
with \cite{Ch4}, Theorem~4.2, proves our assertion.
\end{proof}

\smallskip
Let $T$ be an abelian torsion group and $S _{0}(T) = \{\pi \in
\mathbb P\colon \ T _{\pi } = \{0\}\}$. Assume that $2 \in S
_{0}(T)$ and denote by $S _{1} (T)$ the set of those $p \in \mathbb
P \setminus S _{0}(T)$, for which $S _{0}(T)$ contains the prime
divisors of $p - 1$. Clearly, if $T \neq \{0\}$, then $S _{1}(T)$
contains the least element of $\mathbb P \setminus S _{0}(T)$. Using
Theorem 3.1, Lemma~3.5 (i) and \cite{Ch4}, Theorem~4.2, one obtains
the following results:
\par
\medskip
(5.2) (i) If $T$ is divisible with $T _{p} \cong \mathbb Z(p
^{\infty })$, for every $p \in S _{1}(T)$, then there exists a
Henselian PQL-field $(K, v)$, such that Br$(K) \cong T$, $v(K)$ is
totally indivisible and char$(\widehat K) = 0$;
\par
(ii) If Br$(F) \cong T$, for some Henselian PQL-field $(F, w)$ with
$w(F)$ totally indivisible, then $T _{p} \cong \mathbb Z(p ^{\infty
})$ and $F$ contains a primitive $p$-th root of unity, for each $p
\in S _{1}(T)$, $p \neq {\rm char}(\widehat F)$;
\par
(iii) The group $T$ satisfies the condition in (i) if and only if $T
\cong {\rm Br}(L)$, for some Henselian real-valued PQL-field $(L,
\omega )$.

\medskip
The conclusion of (5.2) (iii) is not necessarily true without the
condition that $\omega (L) \le \mathbb R$. Indeed, let $T$ be an
abelian torsion group with $T _{2} = \{0\}$, $S _{\pi }(T) = S
_{1}(T) \setminus \{\pi \}$, for some $\pi \in S _{1}(T)$, and $S
^{\prime }_{\pi }(T)$ the set of those $p ^{\prime } \in \mathbb P
\setminus S _{0}(T)$, for which the coset of $\pi $ in $\mathbb Z/p
^{\prime }\mathbb Z = \mathbb F _{p'}$ has order in $\mathbb F _{p'}
^{\ast }$ not divisible by any $p \in \mathbb P\setminus S _{\pi
}(T)$. Fix an algebraic closure $\overline {\mathbb Q _{\pi }}$ of
the field $\mathbb Q _{\pi }$ of $\pi $-adic numbers, and for each
subset $\Pi $ of $\mathbb P$, put $\Pi ^{\prime } = \Pi \cup \{\pi
\}$, $\overline {\Pi } = \mathbb P \setminus \Pi $, and denote by
$U _{\Pi }$ the compositum of inertial extensions of $\mathbb Q
_{\pi }$ in $\overline {\mathbb Q _{\pi }}$ of degrees not divisible
by any $\bar {\pi } \in \overline {\Pi }$. It is well-known that $U
_{\pi }/\mathbb Q _{\pi }$ is a Galois extension with $\mathcal{G}(U
_{\pi }/\mathbb Q _{\pi })$ isomorphic to the topological group
product $\prod _{p \in \Pi } \mathbb Z_{p}$. This implies that $U
_{\pi }(p ^{\prime })/\mathbb Q _{\pi }$ is Galois, for each $p
^{\prime } \in \mathbb P$. Observing also that $\mathcal{G}(U _{\Pi
}/\mathbb Q _{\pi })$ is a projective profinite group \cite{G},
Theorem~1, one concludes that, for each $\Pi \subseteq \mathbb P$,
there exists a field $F _{\Pi } \in I(U _{\Pi '} (\pi )/U _{\Pi })$,
such that $U _{\Pi '}.F _{\Pi } = U _{\Pi '}(\pi )$ and $U _{\Pi '}
\cap F _{\Pi } = U _{\Pi }$. This observation, combined with (4.1),
Proposition 2.5 and Theorem 1.1, enables one to obtain the following
result (arguing in the spirit of the proof of (5.2) and Corollary
5.5):
\par
\medskip
(5.3) $T \cong {\rm Br}(K)$, for some Henselian PQL-field $(K, v)$,
such that $v(K)$ is totally indivisible and char$(\widehat K) = \pi
$, if and only if $T$ is divisible with $T _{p'} \cong \mathbb Z(p
^{\prime \infty })$, for every $p ^{\prime } \in S ^{\prime }_{\pi
}$. When $T$ has the noted properties and $T _{\pi } \not\cong
\mathbb Z(\pi ^{\infty })$, char$(K) = 0$ and $G(K) = \pi G(K)$.
Moreover $(K, v)$ can be chosen so that $K _{G(K)} = F _{S _{\pi
}(T)}$ and $\hat v _{G(K)}$ extends the natural valuation of
$\mathbb Q _{\pi }$.

\medskip
Corollary 5.5, statements (5.2) and (5.3), and the classification of
divisible abelian groups (cf. \cite{F}, Theorem~23.1) describe the
isomorphism classes of Brauer groups of Henselian PQL-fields with
totally indivisible value groups.

\section{Applications}

The first result of this Section together with Theorem~2.1 of
\cite{Ch2}, I, characterizes the quasilocal property in the class of
Henselian quasilocal fields with totally indivisible value groups.

\medskip
\begin{prop}
Let $(K, v)$ be a Henselian field, such that {\rm char}$(\widehat K)
= q \neq 0$, and let $K _{p}$ be the fixed field of some Sylow
pro-$p$-subgroup $G _{p} \le \mathcal{G}_{K}$, for each $p \in \Pi
(K)$. Assume that $v(K) \neq pv(K)$, for every $p \in \Pi (K)$, and
$K$ possesses at least one finite extension in $K _{\rm sep}$ of
nontrivial defect. Then $K$ is quasilocal if and only if it
satisfies the following two conditions:
\par
{\rm (i)} $\widehat K$ is perfect, $q \notin \Pi (\widehat K)$,
$v(K)/qv(K)$ is of order $q$, and $K _{\rm sep}$ contains as a
subfield a norm-inertial $\mathbb Z_{q}$-extension $Y$ of $K _{q}$,
such that every finite extension $L _{q}$ of $K _{q}$ in $K _{\rm
sep}$ with $L _{q} \cap Y = K _{q}$ is totally ramified;
\par
{\rm (ii)} $r(p)_{K _{p}} \le 2$, for each $p \in \Pi (K) \setminus
\{q\}$.
\end{prop}

\begin{proof}
The Henselian property of $(K, v)$ ensures that $K _{p} ^{\ast }/K
_{p} ^{\ast p} \cong \widehat K _{p} ^{\ast }/\widehat K _{p} ^{\ast
p} \oplus v(K)/pv(K)$, for each $p \in \mathbb P \setminus \{q\}$.
Since $v(K)/pv(K) \cong v(K _{p})/pv(K _{p})$ and $v(K) \neq pv(K)$,
this enables one to deduce from (2.3) that Br$(K _{p}) \neq \{0\}$
whenever $r(p)_{K _{p}} \ge 2$. These observations, combined with
(1.4), \cite{W}, Lemma~7, \cite{Ch5}, I, Lemma~3.8, indicate that
condition (ii) holds if and only if $K _{p}$ is $p$-quasilocal, for
each $p \in \Pi (K) \setminus \{q\}$. At the same time, it follows
from (2.3) and our assumptions that there exists a finite extension
of $K _{q}$ in $K _{\rm sep}$ of nontrivial defect. Suppose that $K
_{q}$ is quasilocal. Then (4.2), Lemma 4.1 and the preceding
observation imply the existence of an immediate $\mathbb Z
_{q}$-extension $Y$ of $K _{q}$ in $K _{\rm sep}$. In view of
Remark 4.5 and Theorem 4.1, Br$(K _{q}) \cong \mathbb Z(q ^{\infty
})$, so it follows from \cite{P}, Sect. 15.1, Proposition~b and the
Henselian property of $v _{K _{q}}$ that $Y$ is norm-inertial over
$K _{q}$. Since $q$ does not divide the degrees of the finite
extensions of $K$ in $K _{q}$, the fulfillment of the remaining part
of condition (i) can be proved by applying Lemmas 3.3 and 3.5 to $(K
_{q}, v _{K _{q}})$.
\par
Our objective now is to show that $K _{q}$ is $q$-quasilocal,
provided that condition (i) holds. Put $w = v _{K _{q}}$, denote by
$H$ the maximal $q$-divisible group from Is$_{w}(K _{q})$, and for
each $n \in \mathbb N$, let $Y _{n}$ be the extension of $K _{q}$ in
$Y$ of degree $q ^{n}$. Since $\widehat K$ is perfect, it follows
from the assumption on $Y/K _{q}$ and the choice of $H$ that $\mu
\in N(Y _{n}/K _{q})$, for every $n \in \mathbb N$ whenever $\mu \in
K _{q} ^{\ast }$ and $v _{K _{q}}(\mu ) \in H$. At the same time,
the $q$-divisibility of $H$ implies that every finite extension $L
_{q}$ of $K _{q}$ in $K _{\rm sep}$ with $L _{q} \cap Y = K _{q}$ is
totally ramified relative to $w _{H}$. The obtained results enable
one to deduce from Lemma 3.2 that the residue field of $(K _{q}, w
_{H})$ is perfect (of characteristic $q$ or zero). Observing also
that $w _{H}(Y) = w _{H}(K _{q})$, one concludes that $Y$ satisfies
one of the following two conditions:
\par
\medskip
(6.1) (i) $Y$ equals the compositum of the inertial extensions of $K
_{q}$ in $K _{\rm sep}$ relative to $w _{H}$;
\par
(ii) $Y/K _{q}$ is norm inertial relative to $w _{H}$.
\par
\medskip\noindent
The fulfillment of (6.1) guarantees that $K _{q}$ is quasilocal (see
\cite{Ch2}, I, Lemma~2.2, and the proof of \cite{Ch1}, Theorem~3.1
(b) (ii)), so we assume further that $Y/K _{q}$ satisfies (6.1)
(ii). Considering, if necessary, $w _{H}$ instead of $w$, and using
the observations preceding the statement of (6.1), one obtains a
reduction of the proof of the $q$-quasilocal property of $K _{q}$ to
the special case in which every nontrivial group from Is$_{w}(K
_{q})$ is $q$-indivisible. We first show that $_{q}{\rm Br}(K _{q})
= {\rm Br}(Y _{1}/K _{q})$. Let $L$ be a finite extension of $Y$ in
$K _{\rm sep}$. Then it follows from Galois theory and the
projectivity of $\mathbb Z_{q}$ \cite{G}, Theorem~1, that $L
\subseteq Y.L _{q} ^{\prime }$, for some finite extension $L _{q}
^{\prime }$ of $K _{q}$ in $K _{\rm sep}$, such that $L _{q}
^{\prime } \cap Y = K _{q}$. This implies that $L/Y$ is defectless,
so it follows from \cite{TY}, Theorem~3.1, that every $D \in d(Y)$
is defectless. Observing, however, that $\widehat Y$ is perfect, $q
\notin P(\widehat Y)$, and for each $n \in \mathbb N$, $v(Y)/q
^{n}v(Y)$ is a cyclic group of order $q ^{n}$, one concludes that
ind$(D)$ divides the defect of $D$ over $Y$. Thus it turns out that
Br$(Y) = \{0\}$ and Br$(K _{q}) = {\rm Br}(Y/K _{q}) = \cup _{n=1}
^{\infty } {\rm Br}(Y _{n}/K _{q})$. Since $Y/K _{q}$ is
norm-inertial, this enables one to deduce from (3.4), Hilbert's
Theorem 90 and general properties of cyclic algebras (cf. \cite{P},
Corollary~b) that $_{q} {\rm Br}(K _{q}) = {\rm Br}(Y _{1}/K _{q})$,
as claimed. Let now $Y ^{\prime }$ be an extension of $K _{q}$ in $K
_{\rm sep}$, such that $Y ^{\prime } \neq Y _{1}$ and $[Y ^{\prime
}\colon K _{q}] = q$. Then $Y ^{\prime }/K _{q}$ is totally
ramified. Since $\widehat K _{q}$ is perfect and $w(K _{q})/qw(K
_{q})$ is of order $q$, this ensures that $K _{q} ^{\ast } \subseteq
Y ^{\prime \ast p} \nabla _{0}(Y ^{\prime })$. Observing also that
$YY ^{\prime }/Y ^{\prime }$ is norm-inertial (as follows from (3.4)
and (2.3)), one concludes that $K _{q} ^{\ast } \subseteq N(Y _{1}Y
^{\prime }/Y ^{\prime })$. In view of \cite{P}, Sect. 15.1,
Proposition~b, the obtained result implies that Br$(Y _{1}/K _{q})
\subseteq {\rm Br}(Y ^{\prime }/K _{q})$. As Br$(Y ^{\prime }/K
_{q}) \subseteq _{q}{\rm Br}(K _{q})$, this proves that Br$(Y
^{\prime }/K _{q}) = {\rm Br}(Y _{1}/K _{q}) = _{q}{\rm Br}(K)$,
which means that $K _{q}$ is $q$-quasilocal. It is now easy to
deduce from \cite{Ch5}, I, Lemma~8.3, that $K$ is quasilocal when
conditions (i) and (ii) hold, which completes the proof of
Proposition 6.1.
\end{proof}

\medskip
The following result shows that every nontrivial divisible subgroup
of $\mathbb Q/\mathbb Z$ is realizable as a Brauer group of a
quasilocal field of the type characterized by Proposition 6.1. In
view of Corollary 5.3, it clearly illustrates the fact that the
study of Henselian quasilocal fields with totally indivisible value
groups does not reduce to the special case covered by \cite{Ch2},
II, Theorem~2.1.

\begin{prop}
Let $(\Phi , \omega )$ be a Henselian discrete valued field, such
that $\widehat \Phi $ is quasifinite and {\rm char}$(\widehat \Phi )
= q
> 0$, and let $T$ be a divisible subgroup of
$\mathbb Q /\mathbb Z$ with $T _{q} \neq \{0\}$. Then there exists a
Henselian quasilocal field $(K, v)$ with the following properties:
\par
{\rm (i)} {\rm Br}$(K) \cong T$, $K/\Phi $ is a field extension of
transcendency degree $1$, and $v$ is a prolongation of $\omega $;
\par
{\rm (ii)} $v(K)$ is Archimedean and totally indivisible, $\widehat
K/\widehat \Phi $ is an algebraic extension and $K$ possesses an
immediate $\mathbb Z_{q}$-extension.
\end{prop}

The proof of Proposition 6.2 is constructive and can be found in
\cite{Ch2}, II, Sect. 4. Our next result gives a criterion for
divisibility of value groups of Henselian quasilocal fields, and for
defectlessness of their central division algebras.

\medskip
\begin{prop}
Let $E$ be a quasilocal field satisfying one of the following two
conditions:
\par\smallskip
{\rm (i)} Every finite group $G$ is isomorphic to a subquotient of
$\mathcal{G}_{E}$, i.e. to a homomorphic image of some open subgroup
$H _{G}$ of $\mathcal{G}_{E}$;
\par
{\rm (ii)} {\rm Br}$(E _{p}) \neq \{0\}$, for each $p \in \Pi (E)$,
$E _{p}$ being the fixed field of some Sylow pro-$p$-subgroup $G
_{p}$ of $\mathcal{G}_{E}$; moreover, if {\rm Br}$(E _{p}) \cong
\mathbb Z (p ^{\infty })$, then $2 < r(p)_{E _{p}} < \infty $.
\par\smallskip\noindent
Assume that $E$ has a Henselian valuation $v$. Then $v(E)$ is
divisible and, in case {\rm (i)}, every $D \in d(E)$ is inertial
over $E$ relative to $v$.
\end{prop}

\begin{proof} For each $p \in \mathbb P$, $E _{p}$ is $p$-quasilocal,
$v(E _{p})/pv(E _{p}) \cong v(E)/pv(E)$, and $E _{p}$ contains a
primitive $p$-th root of unity unless $p = {\rm char}(E) = p$. These
observations, combined with (1.4), imply that if $v(E) \neq pv(E)$
and $p \neq {\rm char}(\widehat E)$, then $r(p)_{E _{p}} \le 2$.
When condition (ii) holds, this contradicts Theorem 1.1, and thereby
proves that $v(E) = pv(E)$ in case $p \neq {\rm char}(\widehat E)$.
Suppose now that char$(\widehat E) = p$, $p \in \Pi (E)$ and $E$
satisfies condition (ii). We first show that $r(p)_{E _{p}} = \infty
$. Assuming the opposite, one obtains from \cite{J},
Proposition~4.4.8, and \cite{Ch6}, Corollary~5.3, that char$(E) =
0$, Br$(E _{p}) \cong \mathbb Z(p ^{\infty })$ and $r(p)_{E _{p}}
\ge 3$. It is therefore clear from (4.1) that $v(E)/G(E)$ is
$p$-divisible, $G(E)$ being defined as in Section 2. At the same
time, \cite{E1}, Proposition~3.4, and the assumptions on $r(p)_{E
_{p}}$ require that $v(E _{p}) \neq pv(E _{p})$. Thus it turns out
that $v(E) \neq pv(E)$ and $G(E) \neq pG(E)$, so it follows from
Proposition 2.6 that $G(E)$ is cyclic and $v(E)$ is totally
indivisible. The obtained contradiction proves that $r(p)_{E _{p}} =
\infty $. Hence, by condition (ii), $_{p}{\rm Br}(E _{p})$ is
noncyclic, which enables one to deduce from (5.2) (ii) and Theorem
1.1 that $v(E) = pv(E)$. Thus the fulfillment of (ii) guarantees
that $v(E)$ is divisible.
\par
We turn to the proof of the divisibility of $v(E)$ in case (i) of
Proposition 6.3. Our assumptions indicate that, for each $p \in
\mathbb P$, there are finite extensions $L _{p}$, $L _{p} ^{\prime
}$ and $L _{p} ^{\prime \prime }$ of $E$ in $E _{\rm sep}$, such
that $L _{p} ^{\prime }$ and $L _{p} ^{\prime \prime }$ are Galois
over $L _{p}$ with $\mathcal{G}(L _{p} ^{\prime }/L _{p})$ and
$\mathcal{G}(L _{p} ^{\prime \prime }/L _{p})$ isomorphic to the
alternating groups Alt$_{jp}$ and Alt$_{kp}$, respectively, for some
integers $j$ and $k$ with $5 \le j < k$. It is clear from Galois
theory, the choice of $L _{p} ^{\prime }$ and $L _{p} ^{\prime
\prime }$, and the simplicity of the groups Alt$_{n}$, $n \ge 5$,
that $L _{p} ^{\prime }L _{p} ^{\prime \prime }$ is a Galois
extension of $L _{p}$ with $\mathcal{G}(L _{p} ^{\prime }L _{p}
^{\prime \prime }/L _{p}) \cong {\rm Alt}_{jp} \times {\rm
Alt}_{kp}$. Denote by $U _{p}$ the compositum of the inertial
extensions of $L _{p}$ in $E _{\rm sep}$, and put $I _{p} ^{\prime }
= U _{p} \cap L _{p} ^{\prime }$, $I _{p} ^{\prime \prime } = U _{p}
\cap L _{p} ^{\prime \prime }$. Note that $U _{p}/L _{p}$ is Galois
and $\mathcal{G}_{U _{p}}$ is prosolvable (see \cite{JW}, page 135,
and, e.g., \cite{Ch2}, I, page 3102), which means that
$\mathcal{G}(L _{p} ^{\prime }/I _{p} ^{\prime })$ and
$\mathcal{G}(L _{p} ^{\prime \prime }/I _{p} ^{\prime \prime })$ are
solvable groups. Therefore, the preceding observation also shows
that $L _{p} ^{\prime }L _{p} ^{\prime \prime } \subseteq U _{p}$.
Let now $M _{p}$ be the fixed field of some Sylow $p$-subgroup of
$\mathcal{G}(L _{p} ^{\prime }L _{p} ^{\prime \prime }/L _{p})$, and
let $H _{p} ^{\prime }$ and $H _{p} ^{\prime \prime }$ be Sylow
$p$-subgroups of $\mathcal{G}(L _{p} ^{\prime }/L _{p})$ and
$\mathcal{G}(L _{p} ^{\prime \prime }/L _{p})$, respectively. Then
$L _{p} ^{\prime }L _{p} ^{\prime \prime }/M _{p}$ is an inertial
Galois extension with $\mathcal{G}(L _{p} ^{\prime }L _{p} ^{\prime
\prime }/M _{p}) \cong H _{p} ^{\prime } \times H _{p} ^{\prime
\prime }$. The obtained result indicates that $r(p)_{\widehat M
_{p}} \ge 2$. As $M _{p}$ is quasilocal, this yields $v(M _{p}) =
pv(M _{p})$ (see Remark 3.4). Hence, by the isomorphism $v(E)/pv(E)
\cong v(M _{p})/pv(M _{p})$, we have $v(E) = pv(E)$, which proves
that $v(E)$ is divisible, as claimed.
\par
It remains to be seen that the fulfillment of condition (i) of
Proposition 6.3 ensures that every $D \in d(E)$ is inertial. As
$v(E)$ is divisible, it suffices to show that $D/E$ is defectless.
In view of (2.3) and (1.1), one may assume, for the proof, that
char$(\widehat E) = q > 0$ and Br$(E _{q}) \neq \{0\}$. As shown
above, $q \in P(\widehat E _{q})$ and $v(E _{q}) = qv(E _{q})$, so
Lemma 3.6 and our assumptions indicate that finite extensions of $E
_{q}$ in $E _{\rm sep}$ are defectless. It is therefore clear from
(2.3) that finite extensions of $E$ in $E _{\rm sep}$ have the same
property, so the concluding assertion of Proposition 6.3 follows
from \cite{TY}, Theorem~3.1.
\end{proof}

Our next result characterizes $p$-adically closed fields in the
class of Henselian quasilocal fields with residue fields of
characteristic $p$; it essentially gives an answer to a question
posed to the author by Serban Basarab.

\medskip
\begin{prop}
For a Henselian field $(K, v)$ with {\rm char}$(\widehat K) = p >
0$ and $v(K) \neq pv(K)$, the following conditions are equivalent:
\par
{\rm (i)} $(K, v)$ is $p$-adically closed;
\par
{\rm (ii)} $K$ is quasilocal, {\rm char}$(K) = 0$, $r(p)_{K} \in
\mathbb N$ and either $r(p)_{K} \ge 3$ or $\mathcal{G}(K(p)/K)$ is a
free pro-$p$-group with $r(p)_{K} = 2$.
\end{prop}

\begin{proof}
Note that condition (i) of Proposition 2.6 holds when $(K, v)$
satisfies condition (i) or (ii) of Proposition 6.4 (see (4.1),
\cite{W}, page 725, and \cite{PR}, Theorem~3.1). Therefore, one may
assume for the proof that char$(K) = 0$, $\widehat K$ is finite and
the subgroup $G(K) \le v(K)$ is cyclic. This enables one to deduce
from (1.4) and Remark 2.8 that if $K$ is quasilocal, then $v
_{G(K)}(K) = v(K)/G(K)$ is divisible. In view of \cite{PR},
Theorem~3.1, the obtained results prove that (ii)$\to $(i). Assume
now that $(K, v)$ is $p$-adically closed. This ensures that
$v(K)/G(K)$ is divisible, which implies finite extensions of $K$ in
$K _{\rm sep}$ are inertial relative to $v _{G(K)}$. Therefore,
$\mathcal{G}_{K} \cong \mathcal{G}_{K _{G(K)}}$ (see \cite{JW}, page
135), so it follows from \cite{Ch5}, I, Theorem~8.1, that $K$ is
quasilocal if and only if so is $K _{G(K)}$. As $\widehat K$ is
finite and $G(K)$ is cyclic, the assertion that $K _{G(K)}$ is
quasilocal can be deduced from (2.1) (i), Propositions 2.1 and
\cite{Ch2}, I, Corollary~2.5. Finally, the isomorphism
$\mathcal{G}_{K} \cong \mathcal{G}_{K _{G(K)}}$, combined with (2.6)
(ii) and \cite{S1}, Ch. II, Theorems~3 and 4, indicates that $\mathcal{G}(K(p)/K)$ and $r(p)_{K}$ satisfy condition (ii), so the
implication (i)$\to $(ii) and Proposition 6.4 are proved.
\end{proof}

\medskip
The concluding result of this Section supplements Theorem 1.1 by
showing that Brauer groups of quasilocal fields with Henselian
valuations whose residue fields are separably closed and imperfect
do not necessarily embed in $\mathbb Q/\mathbb Z$. In addition to
Proposition 6.3, it simultaneously indicates that central division
algebras over such fields are not necessarily defectless.

\medskip
\begin{prop}
Assume that $(\Phi , \omega )$ is a Henselian field, such that
$\omega (\Phi ) = \mathbb Z$, $\widehat \Phi $ is quasifinite and
{\rm char}$(\widehat \Phi ) = q > 0$. Then there exists a Henselian
field $(K, v)$ with the following properties:
\par
{\rm (i)} $v(K) = \mathbb Q $, {\rm char}$(\widehat K) = q$ and $[\widehat
K\colon \widehat K ^{q}] = q$;
\par
{\rm (ii)} $\mathcal{G}_{K}$ is a pro-$q$-group and $_{q}{\rm Br}(K)
\cong \widehat K ^{\ast }/\widehat K ^{\ast q}$; in particular,
$_{q}{\rm Br}(K)$ is infinite;
\par
{\rm (iii)} $K$ is quasilocal and has a unique immediate $\mathbb Z
_{q}$-extension $I _{\infty }$ in $K _{\rm sep}$; for each finite
extension $L$ of $K$ in $K _{\rm sep}$ of nontrivial defect, $L \cap
I _{\infty } \neq K$;
\par
{\rm (iv)} $K/\Phi $ is a field extension of transcendency degree
$1$.
\end{prop}

\begin{proof}
Let $\overline {\Phi }$ be an algebraic closure of $\Phi $. By the
proof of \cite{Ch2}, II, Theorem~1.2, there exists a field
$\widetilde K \in I(\overline {\Phi }/\Phi )$, such that $\overline
{\Phi }/\widetilde K$ is an immediate $\mathbb Z _{q}$-extension
relative to $v _{\widetilde K}$; in particular, one can apply (3.4)
(iii) to the fields from $I(\overline {\Phi }/\widetilde K)$ the
conditions of \cite{Ch2}, II, Lemma 3.1. Put $\tilde v = v
_{\widetilde K}$ and let $\tilde k$ be the residue field of
$(\widetilde K, \tilde v)$. Denote by $v _{X}$ the Gauss valuation
of the rational function field $\widetilde K(X)$, which extends
$\tilde v$ so that $v _{X}(f(X)) = 0$, for each $f(X) \in O _{\tilde
v}[X] \setminus M _{\tilde v}[X]$, and fix a Henselization $(F, v)$
of $(\widetilde K(X), v _{X})$. The definition of $v _{X}$ shows
that the residue class of $X$ is transcendental over $\tilde k$. Let
$G _{q}$ be a Sylow pro-$q$-subgroup of $\mathcal{G}_{F}$, $K$ the
fixed field of $G _{q}$ and $v = w _{K}$. Arguing in the spirit of
the proof of the quasilocal property of the field $K _{q}$
(considered in \cite{Ch2}, II, Sect. 4), one obtains first that
Br$(\widetilde K _{1}K/K) = _{q}{\rm Br}(K)$, where $\widetilde K
_{1}$ is the extension of $\widetilde K$ in $\Gamma $ of degree $q$.
This is used for proving that $K$, $v$ and $I _{\infty } = \overline
{\Phi }K$ have the properties required by Proposition 6.5.
\end{proof}

\medskip
\begin{rema}
Let $F _{0}$ be a global field, $w _{0}$ a discrete valuation of $F
_{0}$, $q = {\rm char}(\widehat F _{0})$, $F$ a completion of $F
_{0}$ with respect to $\omega _{0}$, and $w$ the valuation of $F$
continuously extending $w _{0}$. It is well-known that tr$(F/F
_{0})$ is (uncountably) infinite. Fix a purely transcendental
extension $F _{n}$ of $F _{0}$ in $F$ so that tr$(F _{n}/F _{0}) = n
\ge 0$, $n \le \infty $, denote by $\Phi _{n}$ the separable closure
of $F _{n}$ in $F$, and let $\omega _{n}$ be the valuation of $\Phi
_{n}$ induced by $w$. Then $\omega _{n}$ is discrete and Henselian
and $\widehat \Phi _{n} \cong \mathbb F _{q}$. Therefore, one can
find an extension $R _{n}$ of $\Phi _{n}$ with the properties
required by Proposition 6.5. When $n \in \mathbb N$, our
construction ensures that tr$(R _{n-1}/F _{0}) = n$.
\end{rema}

\medskip
For each $m \in \mathbb N$, the quasilocal fields $R _{n}$, $n \in
\mathbb N$, in Remark 6.6 have infinitely many nonisomorphic
algebras $D _{n,m} \in d(R _{n})$ of index $p ^{m}$. By \cite{Ch3},
Theorem~4.1 and Corollary~8.6, for each admissible pair $(n, m)$,
all $D _{n,m}$ share, up-to $R _{n}$-isomorpisms, a common set of
maximal subfields, and a common class of splitting fields algebraic
over $R _{n}$. Since $R _{n}$ is of transcendency degree $n + 1$ or
$n + 2$ over its prime subfield, for each $n < \infty $, this raises
interest in the open problem of whether there exist finitely
generated fields $F$ which possess infinitely many nonisomorphic $D
\in d(F)$ with some of the noted two properties of the algebras $D
_{m,n}$ (see \cite{KMc}, and for the case of quaternion algebras
\cite{RR}, \cite{GS} and \cite{PraRap}, Remark~5.4). The
corresponding problem for arbitrary fields has an affirmative
solution (found by Van den Bergh-Schofield \cite{VdBS}, Sect. 3, and
Saltman, see \cite{GSz}, Sect. 5.5). In view of \cite{Ch5}, I,
Corollary~8.5, a complete solution to the general problem is
obtained by applying the latter assertion of (1.3) (i), to a field
$E _{0}$ of zero characteristic and to a divisible abelian torsion
group $T$ with infinite components $T _{p}$, for all $p \in \mathbb
P$.

\end{document}